\documentclass[a4paper,12pt]{article}
\usepackage{amsmath} \usepackage{amssymb} 
\usepackage{eucal}
\usepackage{geometry}
\usepackage{graphicx}
\usepackage[amsmath,thmmarks]{ntheorem}
\theoremstyle{change}
\theoremseparator{.}

\theorembodyfont{\normalfont\slshape}
\newtheorem{theorem}{Theorem}[section]

\newtheorem{lemma}[theorem]{Lemma}
\newtheorem{proposition}[theorem]{Proposition}

\theorembodyfont{\normalfont}

\def\proofsymbol{\rule{0.5em}{0.5em}}

\theoremheaderfont{\normalfont\itshape\bfseries}
\theoremstyle{nonumberplain}
\theoremsymbol{\enspace\ensuremath{{\proofsymbol}}}
\newtheorem{proof}{Proof}

\theoremstyle{empty}
\newtheorem{proofof}{}

\usepackage[shortlabels]{enumitem}
\setlist{
  listparindent=\parindent,
  parsep=0pt,
  itemsep=0.5em plus 0.25em minus 0.2em,
}

\setenumerate{
leftmargin=*,
labelsep=0.5em,
}
\setitemize{
leftmargin=*,
labelsep=0.5em,
}

\SetEnumerateShortLabel{s}{\textup{($\sigma$\arabic*)}}
\SetEnumerateShortLabel{a}{\textup{(\alph*)}}
\SetEnumerateShortLabel{A}{\textup{(\Alph*)}}
\SetEnumerateShortLabel{i}{\textup{(\roman*)}}
\SetEnumerateShortLabel{I}{\textup{(\Roman*)}}
\SetEnumerateShortLabel{m}{\textup{(m\arabic*)}}
\SetEnumerateShortLabel{d}{\textup{($\delta$\arabic*)}}
\SetEnumerateShortLabel{p}{\textup{(pm\arabic*)}}
\SetEnumerateShortLabel{1}{\textup{(\arabic*)}}
\SetEnumerateShortLabel{j}{\textup{(i\arabic*)}}
\SetEnumerateShortLabel{n}{\textup{(n\arabic*)}}

\newlist{lenumerate}{enumerate}{3}
\setlist[lenumerate]{
  listparindent=\parindent,
  itemsep=0.5em plus 0.25em minus 0.3em,
  fullwidth,
  labelsep=0.75em,
  label=\arabic*.
}

\numberwithin{equation}{section} 

\def\C{{\mathbb C}}

\def\N{{\mathbb N}}
\def\R{{\mathbb R}}

\def\CC{{\mathcal C}}

\def\CG{{\mathcal G}}

\def\CL{{\mathcal L}}

\def\CP{{\mathcal P}}

\def\CR{{\mathcal R}}
\def\CS{{\mathcal S}}

\def\phi{\varphi}

\def\qand{{\quad\mbox{and}\quad}}

\def\d{{\rm d}}
\def\1{{{\ae}}}
\def\2{{{\o}}}
\def\3{{{\aa}}}
\def\6{\, {\rm d}}
\def\ri{{\rm i}}
\def\e{{\rm e}}
\def\<{{\langle}}
\def\>{{\rangle}}
\def\brinv{{\langle -1\rangle}}

\def\im{{\sf Im}}
\def\re{{\sf Re}}
\def\supp{{\sf supp}}
\def\ID{\mathcal{ID}}
\def\diff{\frac{\rm d}{{\rm d}t}_{\bigm |_{t=0}}}

\begin{document}

\title{On the free Gamma distributions  \date{}}

\author{\sc Uffe Haagerup\footnote{Uffe Haagerup is supported by ERC
    Advanced Grant No.~OAFPG~27731 and the Danish National Research
    Foundation through the Center for Symmetry and Deformation (DNRF92).}
 and Steen Thorbj{\o}rnsen}

\maketitle

\begin{abstract}
For each positive number $\alpha$ we study the analog $\nu_\alpha$ in free
probability of the classical Gamma distribution with parameter
$\alpha$. We prove that $\nu_\alpha$ is absolutely continuous and
establish the main properties of the density, including analyticity
and unimodality. We study further the
asymptotic behavior of $\nu_\alpha$ as $\alpha\downarrow 0$.
\end{abstract}

\section{Introduction}

In this paper we study the free Gamma distributions, i.e., the images
of the classical Gamma distributions under the bijection between the
classes of infinitely divisible measures in classical and free
probability, respectively,
introduced by Bercovici and Pata (cf.\ \cite{bp2} and
\cite{bt1}). More precisely, for any positive number $\alpha$
the free gamma distribution $\nu_\alpha$ with parameter $\alpha$ is
defined as $\Lambda(\mu_\alpha)$, where $\Lambda$ is the
Bercovici-Pata bijection (see Section~\ref{prelim}) and $\mu_\alpha$
is the classical Gamma distribution with parameter $\alpha$, i.e.,
\begin{equation}
\mu_\alpha(B)=\frac{1}{\Gamma(\alpha)}
\int_{B\cap[0,\infty)}t^{\alpha-1}\e^{-t}\6t
\label{eqI.4}
\end{equation}
for any Borel set $B$ in $\R$.

The classical Gamma distributions form perhaps the simplest class of
selfdecomposable measures on $\R$ which are not stable (see
Section~\ref{prelim}). Since
$\Lambda$ preserves the notions of stability and selfdecomposability
(see \cite{bp2} and \cite{bt1}), the measures $\nu_\alpha$ are thus of
interest as (the simplest?) examples of non-stable selfdecomposable
measures with respect to free (additive) convolution. Of particular
interest is the free $\chi^2$-distribution $\Lambda(\chi_1^2)$, which
(up to scaling by 2) equals the measure $\nu_{1/2}$. Apart from the
general importance of the $\chi^2$-distribution in classical
probability, this is mainly due to the fact that the 
square of the semi-circle distribution (the analog of the Gaussian
distribution in free probability) equals the free Poisson distribution
(the image of the classical Poisson law by $\Lambda$) as observed e.g.\
in \cite{vdn}. Since $\Lambda$ is injective, the relationship between
the Gaussian and the $\chi^2$-distribution thus breaks down in free
probability, and from that point of view it
is of some interest to identify further the measure $\Lambda(\chi_1^2)$.

In an appendix to the paper \cite{bp2} P.~Biane studied the freely
stable distributions and established their absolute
continuity (with respect to Lebesgue measure) as well as the main
features of their densities; in particular analyticity and
unimodality. Applying the 
same method as Biane (based on Stieltjes inversion) we establish in
the present paper that the free
Gamma distributions $\nu_\alpha$ are absolutely continuous with
analytic densities, and that
they have supports in the form $[s_\alpha,\infty)$ for some strictly
positive number $s_\alpha$, which increases (strictly) with $\alpha$
and tends to 0 and $\infty$ as $\alpha$ goes to $0$ and $\infty$,
respectively. We derive an (implicit) expression for the density
$f_\alpha$ of $\nu_\alpha$ in the form:
\begin{equation}
f_\alpha(P_\alpha(x))=
\frac{1}{\pi}
\frac{v_\alpha(x)}{x^2+v_\alpha(x)^2}
\qquad(x\in[-c_\alpha,\infty)),
\label{eqI.1}
\end{equation}
where $P_\alpha$ is a strictly increasing function given by
\[
P_\alpha(x)=
2x+\alpha-\alpha\int_0^\infty\frac{t^2\e^{-t}}{(x-t)^2+v_\alpha(x)^2}\6t,
\qquad(x\in[-c_\alpha,\infty)),
\]
and $c_\alpha$ is a positive constant such that
$P_\alpha(-c_\alpha)=s_\alpha$. Moreover $v_\alpha\colon\R\to\R$ is a
function essentially defined by the condition:
\[
(x+\ri v_\alpha(x))\big[1+\alpha G_{\mu_1}(x+\ri v_\alpha(x))\big]\in\R,
\]
where $G_{\mu}$ denotes the Cauchy (or Stieltjes) transform of a
probability measure $\mu$ on $\R$ (see formula \eqref{eqPL.6} below).
This condition emerges naturally from the method of Stieltjes
inversion in combination with the key formula:
\begin{equation}
G_{\nu_\alpha}(z(1+\alpha G_{\mu_1}(z)))=\frac{1}{z},
\label{eqI.2}
\end{equation}
which holds for all $z$ in $\C^+$
satisfying that $z(1+\alpha G_{\mu_1}(z))\in\C^+$.
The passage from \eqref{eqI.2} to \eqref{eqI.1} via Stieltjes
inversion depends heavily on a fundamental result of
Bercovici and Voiculescu, which we state in Lemma~\ref{BV key lemma}
below  for the readers
convenience. 

By careful studies of the functions $v_\alpha$, $P_\alpha$
and the right hand side of \eqref{eqI.1}, we derive
some main features of the density $f_{\alpha}$, e.g.\ analyticity,
unimodality and the asymptotic behavior:
\begin{equation}
\lim_{\xi\to\infty}\frac{f_\alpha(\xi)}{\xi^{-1}\e^{-\xi}}=\alpha\e^\alpha, 
\qand
\lim_{\xi\downarrow s_\alpha}\Big(\frac{f_\alpha(\xi)}{\sqrt{\xi-s_\alpha}}\Big)
=\frac{\sqrt{2}}{\pi c_\alpha\sqrt{s_\alpha-c_\alpha^2}}.
\label{eqI.3}
\end{equation}
In particular it follows that $\nu_{\alpha}$ has moments of all
orders, which is in concordance with the results of Benaych-George in
\cite{BG}. 

We study also the asymptotic behavior of $\nu_\alpha$ as
$\alpha\downarrow0$, and we prove that the measures $\frac{1}{\alpha}\nu_\alpha$
converge to the measure $x^{-1}\e^{-x}1_{(0.\infty)}(x)\6x$ in moments
and in the sense of point-wise convergence of the densities:
\[
\lim_{\alpha\downarrow0}\frac{f_\alpha(\xi)}{\alpha}=\xi^{-1}\e^{-\xi}
\qquad(\xi\in(0,\infty)).
\]


The remainder of the paper is organized as follows: In
Section~\ref{prelim} we collect background material on infinite
divisibility, The Bercovici-Pata bijection and Stieltjes inversion. In
Section~\ref{abs cont} we establish absolute continuity of
$\nu_\alpha$ and prove the expression \eqref{eqI.1} for the
density. In Section~\ref{sec asympt opf} we establish the asymptotic
behavior \eqref{eqI.3} and study how the quantities $c_\alpha$ and
$s_\alpha$ vary as functions of $\alpha$. In Section~\ref{sec
  unimodal} we prove that $\nu_\alpha$ is unimodal, and in the final
Section~\ref{sec_alpha_mod_0} we study the asymptotic behavior of
$\nu_\alpha$ as $\alpha\downarrow0$. The main results in
Sections~\ref{abs cont}-\ref{sec_alpha_mod_0} depend in part on some
basic properties of the functions $v_\alpha$ and $P_\alpha$, the
proofs of which are (not surprisingly) rather technical. In order to
maintain a steady flow in the paper, these proofs are deferred to
Appendix~\ref{tekniske beviser} at the end of the paper.

\section{Background}\label{prelim}

\subsection{Classical and free infinite divisibility}\label{PL free ID}

A (Borel-) probability measure $\mu$ on
${\mathbb R}$ is called infinitely divisible, if there exists, for each
positive integer $n$, a probability measure $\mu_n$ on $\R$, such that
\begin{equation}
\mu=\underbrace{\mu_n*\mu_n*\cdots*\mu_n}_{n \ \textrm{terms}},
\label{eqPL.1}
\end{equation}
where $*$ denotes the usual convolution of probability measures (based
on classical independence). We denote by $\ID(*)$ the class of all
such measures on $\R$. 

We recall that a probability measure $\mu$ on ${\mathbb R}$
is infinitely divisible, if and only if its
characteristic function (or Fourier transform)
$\hat{\mu}$ has the L\'evy-Khintchine representation:
\begin{equation}
\hat{\mu}(u)=\exp\Big[{\rm i}\eta u - {\textstyle\frac{1}{2}}au^2 + 
\int_{{\mathbb R}}\big({\rm e}^{{\rm i}ut}-1-{\rm i}ut 1_{[-1,1]}(t)\big) \
\rho({\rm d}t)\Big], \qquad (u\in{\mathbb R}),
\label{e0.10b}
\end{equation}
where $\eta$ is a real constant, $a$ is a non-negative constant and
$\rho$ is a L\'evy measure on ${\mathbb R}$, meaning that
\[
\rho(\{0\})=0, \quad \textrm{and} \quad \int_{{\mathbb R}}\min\{1,t^2\} \
\rho({\rm d}t)<\infty.
\]
The parameters $a$, $\rho$ and $\eta$
are uniquely determined by $\mu$ and the triplet $(a,\rho,\eta)$ is called the
{\it characteristic triplet} for $\mu$.

For two probability measures $\mu$ and $\nu$ on $\R$, the free
convolution $\mu\boxplus\nu$ is defined as the distribution of $x+y$,
where $x$ and $y$ are \emph{freely independent} (possibly unbounded)
selfadjoint operators on a Hilbert space with spectral distribution
$\mu$ and $\nu$, respectively (see \cite{BV} for further details).
The class $\ID(\boxplus)$ of infinitely divisible probability measures
with respect to free convolution $\boxplus$ is defined by replacing
classical convolution $*$ by free convolution $\boxplus$ in
\eqref{eqPL.1}. 

For a (Borel-) probability measure $\mu$ on $\R$ with support
$\supp(\mu)$, the Cauchy (or Stieltjes) transform is the mapping
$G_\mu\colon\C\setminus\supp(\mu)\to\C$ defined by: 
\begin{equation}
G_\mu(z)=\int_{\R}\frac{1}{z-t}\,\mu(\d t),
\qquad(z\in\C\setminus\supp(\mu)).
\label{eqPL.6}
\end{equation}
The \emph{free cumulant transform} $\CC_\mu$ of $\mu$ is then given by
\begin{equation}
\CC_\mu(z)=zG_\mu^\brinv(z)-1
\label{eqPL.6a}
\end{equation}
for all $z$ in a certain region $R$ of $\C^-$ (the lower half complex
plane), where the (right) inverse 
$G_\mu^\brinv$ of $G_\mu$ is well-defined. Specifically $R$ may be
chosen in the form:
\[
R=\{z\in\C^-\mid \tfrac{1}{z}\in\Delta_{\eta,M}\}, \quad\text{where}\quad
\Delta_{\eta,M}=\{z\in\C^+\mid |\re(z)|<\eta\im(z), \ \im(z)>M\}
\]
for suitable positive numbers $\eta$ and $M$.
It was proved in \cite{BV} (see also
\cite{ma} and \cite{vo2}) that
$\CC_\mu$ constitutes the free analog of $\log
f_\mu$ in the sense that it linearizes free convolution:
\[
\CC_{\mu\boxplus\nu}(z)=\CC_\mu(z)+\CC_{\nu}(z)
\]
for all probability measures $\mu$ and $\nu$ on $\R$ and all $z$ in a
region where all three transforms are defined. 
The results in \cite{BV} are presented in terms of a variant,
$\phi_\mu$, of $\CC_\mu$, which is often referred to as the Voiculescu
transform, and which is again a variant of the $R$-transform $R_\mu$
introduced in \cite{vo2}. The relationship is the following: 
\begin{equation}
\phi_{\mu}(z)=\CR_\mu(\tfrac{1}{z})=z\CC_{\mu}(\tfrac{1}{z})
\label{eqPL.6b}
\end{equation}
for all $z$ in a region $\Delta_{\eta,M}$ as above.
In \cite{BV} it was
proved additionally that $\mu\in\ID(\boxplus)$, if and only if
there exists $a$ in $[0,\infty)$, $\eta$ in $\R$ and a L\'evy
  measure $\rho$, such that $\CC_\mu$ has the \emph{free
    L\'evy-Khintchine representation}:
\begin{equation}
\mathcal{C}_{\mu}(z) = \eta z+ az^2 + 
\int_{{\mathbb R}}\Big(\frac{1}{1-tz}-1-tz1_{[-1,1]}(t)\Big)
\ \rho({\rm d}t).
\label{eqPL.2}
\end{equation}
(cf.\ also \cite{bt02b}).
In particular it follows for $\mu$ in $\ID(\boxplus)$
that $\CC_\mu$ can be extended to an
analytic map (also denoted $\CC_\mu$) defined on all of $\C^-$.
The triplet $(a,\rho,\eta)$ is uniquely determined and is
called the {\it free characteristic triplet} for $\mu$.

It was proved in
\cite[Proposition~5.12]{BV} that any measure $\nu$ in $\ID(\boxplus)$ has at
most one atom. In fact the proof of that proposition reveals that an atom $a$ 
for $\nu$ is necessarily equal to the non-tangential limit of
$\phi_{\nu}(z)$ as $z\to0$, $z\in\C^+$. We say that a function
$u\colon\C^+\to\C$ has a non-tangential limit $\ell$ at 0, if for any
positive number $\delta$ we have that
\begin{equation}
\ell=\lim_{z\to0,z\in\triangle_\delta}u(z),
\quad\text{where}\quad
\triangle_\delta=\{z\in\C^+\mid \im(z)>\delta|\re(z)|\}.
\label{eq1.1}
\end{equation}
In order to derive non-tangential limits, the following lemma
(Lemma~5.11 in \cite{BV}) is extremely useful:

\begin{lemma}[\cite{BV}]\label{BV key lemma}
Let $u\colon\C^+\to\C^+$ be an analytic function, and let $\Gamma$ be
a curve in $\C^+$ which approaches 0 non-tangentially. 

If $\lim_{z\to0, z\in\Gamma}u(z)=\ell$, then $\lim_{z\to0,
  z\in\triangle_\delta}u(z)=\ell$ for any positive number $\delta$,
i.e., $u$ has non-tangential limit $\ell$ at $0$.
\end{lemma}

\subsection{The Bercovici-Pata bijection}

In \cite{bp2} Bercovici and Pata introduced a bijection
$\Lambda$ between the two classes $\ID(*)$ and $\ID(\boxplus)$, which
may formally be defined as the mapping sending a measure $\mu$ from
$\ID(*)$ with characteristic triplet $(a,\rho,\eta)$ onto the measure
$\Lambda(\mu)$ in $\ID(\boxplus)$ with \emph{free} characteristic triplet
$(a,\rho,\eta)$. It is then obvious that $\Lambda$ is a bijection, and
it turns out that $\Lambda$ further enjoys the following properties (see
\cite{bp2} and \cite{bt1}): 

\begin{enumerate}[a]

\item If $\mu_1,\mu_2\in{\mathcal{ID}}(*)$, then
  $\Lambda(\mu_1*\mu_2)=\Lambda(\mu_1)\boxplus\Lambda(\mu_2)$.

\item If $\mu\in{\mathcal{ID}}(*)$ and $c\in{\mathbb R}$, then
  $\Lambda(D_c\mu)=D_c\Lambda(\mu)$, where e.g.\ $D_c\mu$ is the
  transformation of $\mu$ by the mapping $x\mapsto cx\colon\R\to\R$.

\item For any constant $c$ in ${\mathbb R}$ we have
  $\Lambda(\delta_c)=\delta_c$, where $\delta_c$ denotes Dirac
  measure at $c$.

\item $\Lambda$ is a homeomorphism with respect to weak convergence.

\end{enumerate}

Most of these properties can be established rather easily from the
following convenient formula:
\begin{equation}
\CC_{\Lambda(\mu)}(\ri z)=\int_0^{\infty}\log\hat{\mu}(zx)\e^{-x}\6x,
\qquad(z\in(-\infty,0), \ \mu\in\ID(*)),
\label{eqPL.7}
\end{equation}
which was derived in \cite{BNT}. The properties (a)-(c) imply that
$\Lambda$ preserves e.g.\ the classes of stable and selfdecomposable
measures. Specifically, let $\CP$ denote the class of all (Borel-)
probability measures on 
$\R$, and recall then that a measure $\mu$ from $\CP$ is called
\emph{stable}, if it satisfies the condition:
\begin{equation}
\forall \alpha,\alpha'>0\ \exists \alpha''>0\
\exists \beta\in\R\colon
D_\alpha\mu*D_{\alpha'}\mu=D_{\alpha''}\mu*\delta_{\beta}.
\label{eqPL.3}
\end{equation}
Recall also that $\mu$ is \emph{selfdecomposable}, if
\begin{equation}
\forall c\in(0,1)\ \exists\mu_c\in\CP\colon\mu=D_c\mu*\mu_c.
\label{eqPL.4}
\end{equation}
Denoting by $\CS(*)$ and $\CL(*)$ the classes of stable and
selfdecomposable measures, respectively, it is well-known 
(see e.g.\ \cite{Sa}) that
$\CS(*)\subseteq\CL(*)\subseteq\ID(*)$. The classes $\CS(\boxplus)$
and $\CL(\boxplus)$ are defined be replacing classical convolution $*$
by free convolution $\boxplus$ in \eqref{eqPL.3}-\eqref{eqPL.4}
above. It was shown in \cite{BV} and \cite{bt1} that
$\CS(\boxplus)\subseteq\CL(\boxplus)\subseteq\ID(\boxplus)$. By
application of properties (a)-(c) above, it follows then easily that
\begin{equation}
\Lambda(\CS(*))=\CS(\boxplus), \qand \Lambda(\CL(*))=\CL(\boxplus)
\label{eqPL.5}
\end{equation}
(see \cite{BV} and \cite{bt1}). 
The measures in $\CS(*)$ may alternatively by characterized as those
measures in $\ID(*)$ whose L\'evy measure has the form
\[
\rho(\d t)=
\big(c_-|t|^{-1-a_-}1_{(-\infty,0)}(t)+c_+t^{-1-a_+}1_{(0,\infty)}(t)\big)\6t
\]
for suitable numbers $c_+,c_-$ in $[0,\infty)$ and $a_+,a_-$ in
  $(0,2)$. Similarly $\CL(*)$ may be characterized as
  the class of measures in $\ID(*)$ with L\'evy measures in the
  form: $\rho(\d t)=|t|^{-1}k(t)\6t$, where $k\colon\R\setminus\{0\}\to\R$
  is increasing on $(-\infty,0)$ and decreasing on $(0,\infty)$.
By the definition of $\Lambda$ and \eqref{eqPL.5} we have the exact
same characterizations of the measures in $\CS(\boxplus)$ and
$\CL(\boxplus)$, respectively, if we let the term ``L\'evy measure'' 
refer to the free L\'evy-Khinthcine representation
\eqref{eqPL.2} rather than the classical one \eqref{e0.10b}.

For any positive number $\alpha$, the classical Gamma distribution
$\mu_\alpha$ with parameter $\alpha$ (cf.\ \eqref{eqI.4}) has L\'evy
measure 
\begin{equation*}
\rho_\alpha(\d t)=\alpha t^{-1}\e^{-t}1_{(0,\infty)}(t)\6t,
\end{equation*}
and thus $\mu_\alpha\in\CL(*)\setminus\CS(*)$. The corresponding free
Gamma distribution, $\nu_\alpha=\Lambda(\mu_\alpha)$, satisfies
accordingly that
$\nu_\alpha\in\CL(\boxplus)\setminus\CS(\boxplus)$. As mentioned in
the introduction, the purpose of the present paper is to disclose
the main features of $\nu_\alpha$ for any $\alpha$ in $(0,\infty)$.

\subsection{Stieltjes inversion.}\label{PL Stieltjes inversion}

Let $\mu$ be a (Borel-) probability measure on $\R$, and consider its
cumulative distribution function:
\[
F_\mu(t)=\mu((-\infty,t]), \qquad(t\in\R),
\]
as well as its Lebesgue decomposition:
\[
\mu=\rho+\sigma,
\]
where the measures $\rho$ and $\sigma$ are, respectively, absolutely
continuous and singular with respect to Lebesgue measure $\lambda$ on
$\R$. It follows from De la Vall\'e Poussin's Theorem (see
\cite[Theorem IV.9.6]{Sak}) that $\rho$ and $\sigma$ may be identified
with the restrictions of $\mu$ to the sets
\[
D_1=\big\{x\in\R\bigm|\textstyle{\lim_{h\to0}}\tfrac{F_\mu(x+h)-F_\mu(x)}{h} 
\ \text{exists in $\R$}\big\}
\]
and
\[
D_\infty=\big\{x\in\R\bigm| \textstyle{\lim_{h\to0}}
\tfrac{F_\mu(x+h)-F_\mu(x)}{h}=\infty\big\},
\]
respectively. In addition we have that (see e.g.\ Theorem~3.23 and
Proposition~3.31 in \cite{Fo})
\[
\lambda(\R\setminus D_1)=0, \qand \rho(\d t)=F_\mu'(t)1_{D_1}(t)\6t,
\]
where, for any $t$ in $D_1$, $F_\mu'(t)$ denotes the derivative of
$F_\mu$ at $t$.

Consider now additionally the Cauchy (or Stieltjes) transform $G_\mu$
defined in \eqref{eqPL.6}. It follows then from general theory of
Poisson-Stieltjes integrals (see \cite{Do}) that
\[
F_\mu'(x)=-\frac{1}{\pi}\lim_{y\downarrow0}\im(G_{\mu}(x+\ri y))
\qquad\text{for all $x$ in $D_1$},
\]
and that
\[
\lim_{y\downarrow0}\big|\im(G_{\mu}(x+\ri y))\big|=\infty
\qquad\text{for all $x$ in $D_\infty$}.
\]
In particular we may conclude that the singular part $\sigma$ of $\mu$
is concentrated on the set
\begin{equation*}
\big\{x\in\R\bigm| 
\textstyle{\lim_{y\downarrow0}|G_{\mu}(x+\ri y)|
=\infty\big\}}
\end{equation*}
(see also Chapter~XIII in \cite{RS}).

\section{Absolute continuity of $\nu_\alpha$}\label{abs cont}

In this section we establish absolute continuity of the free Gamma
distributions $\nu_\alpha$, $\alpha>0$, and prove the formula
\eqref{eqI.1} for the densities. Our starting point is the derivation
of the formula \eqref{eqI.2}, and we introduce for that purpose the
function $H_\alpha\colon\C\setminus[0,\infty)\to\C$ given by 
\begin{equation}
H_\alpha(z)=z+z\alpha G_{\mu_1}(z)
=z+z\alpha\int_0^\infty\frac{\e^{-t}}{z-t}\6t
=z+\alpha+\alpha\int_0^\infty\frac{t\e^{-t}}{z-t}\6t
\label{eq3.12}
\end{equation}
for $z$ in $\C\setminus[0,\infty)$.
By differentiation under the integral sign, note that $H_\alpha$ is
analytic on $\C\setminus[0,\infty)$ with derivatives given by
\begin{equation}
\begin{split}
H_\alpha'(z)&=1-\alpha\int_0^\infty\frac{t\e^{-t}}{(z-t)^2}\6t,
\qquad(z\in\C\setminus[0,\infty)),
\\[.2cm]
H_\alpha^{(k)}(z)&=(-1)^k\alpha k!\int_0^\infty\frac{t\e^{-t}}{(z-t)^{k+1}}\6t,
\qquad(z\in\C\setminus[0,\infty), \ k\in\{2,3,4,\ldots\}).
\label{eq3.9}
\end{split}
\end{equation}
In the following we consider in addition the function
$F\colon\C\setminus[0,\infty)\to(0,\infty)$ given by
\begin{equation}
F(x+\ri y)=\int_0^\infty\frac{t\e^{-t}}{|x+\ri y-t|^2}\6t
=\int_0^{\infty}\frac{t\e^{-t}}{(x-t)^2+y^2}\6t 
\label{eq3.0}
\end{equation}
for all $x,y\in\R$ such that $x+\ri y\in\C\setminus[0,\infty)$.

\pagebreak

\begin{lemma}\label{intro v-alpha}
Let $\alpha$ be a positive number.

\begin{enumerate}[i]

\item There exists a unique positive real number $c_\alpha$ such that
\begin{equation}
\frac{1}{\alpha}=F(-c_\alpha)
=\int_0^{\infty}\frac{t\e^{-t}}{(c_\alpha+t)^2}\6t.
\label{eq3.10}
\end{equation}
The number $c_\alpha$ increases with $\alpha$, and satisfies that
\[
\lim_{\alpha\to0}c_\alpha=0, \qand
\lim_{\alpha\to\infty}c_\alpha=\infty.
\]

\item  There is a function $v_\alpha\colon\R\to[0,\infty)$, such that
\begin{equation}
\{z\in\C^+\mid H_\alpha(z)\in\C^+\}
=\{x+\ri y\mid x,y\in\R, \ y>v_\alpha(x)\}.
\label{eq3.20}
\end{equation}
The function $v_\alpha$ is given by
\begin{align}
v_\alpha(x)&=0, \quad\text{if $x\in(-\infty,-c_{\alpha}]$,}
\label{eq3.6a}
\\[.2cm]
F(x+\ri v_\alpha(x))&=\frac{1}{\alpha}, \quad\text{if
  $x\in(-c_{\alpha},\infty)$.} 
\label{eq3.6}
\end{align}

\item For all $x$ in $\R$ we have that $H_\alpha(x+\ri v_\alpha(x))\in\R$.

\item The function $v_\alpha$ satisfies that $v_\alpha(x)>0$ for all
  $x$ in $(-c_\alpha,\infty)$.

\end{enumerate}
\end{lemma}

\bigskip
\noindent
\[
\includegraphics[height=8.0cm]{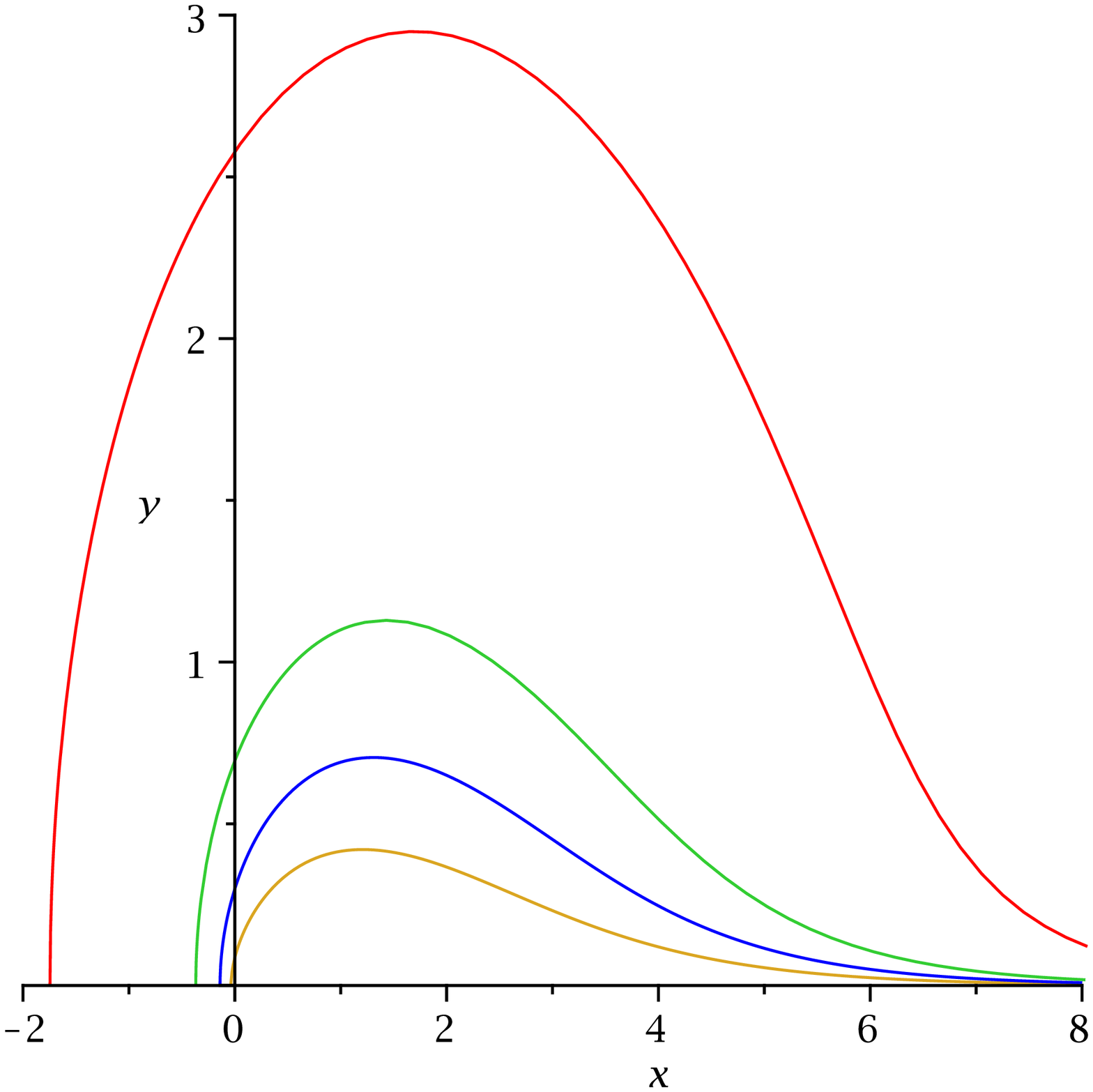}
\]
\begin{center}
{\footnotesize The graphs of the functions $v_{1/2}$, $v_1$, $v_2$ and
  $v_{10}$.}
\end{center}

\begin{proofof}[Proof of Lemma~\ref{intro v-alpha}.] 
(i) \ The function 
$x\mapsto F(-x)=\int_0^\infty\tfrac{t\e^{-t}}{(x+t)^2}\6t$ is
  clearly strictly decreasing and (by dominated
  convergence) continuous on $(0,\infty)$. Moreover, by monotone
  convergence, 
\[
\lim_{x\downarrow0}F(-x)=\infty, \qand \lim_{x\to\infty}F(-x)=0.
\]
Hence, there exists a unique number $c_\alpha$ in $(0,\infty)$ such
that $F(-c_\alpha)=\frac{1}{\alpha}$. The last assertions in (i) are
immediate from the last equality in \eqref{eq3.10}.

(ii) \ For $x,y$ in $\R$ such that $x+\ri y\in\C\setminus[0,\infty)$
  we find from formula \eqref{eq3.12} that
\begin{equation}
\begin{split}
\im(H_\alpha(x+\ri y))
&=y+\alpha\im\Big(\int_0^\infty\frac{t\e^{-t}}{x+\ri y-t}\6t\Big)
\\[.2cm]
&=y-\alpha y\int_0^\infty\frac{t\e^{-t}}{(x-t)^2+y^2}\6t
=y(1-\alpha F(x+\ri y)).
\label{eq3.21}
\end{split}
\end{equation}
For fixed $x$ in $\R$ the function $y\mapsto F(x+\ri
y)$ is clearly strictly decreasing on $(0,\infty)$, and $F(x+\ri y)\to0$
as $y\to\infty$. Moreover, by monotone
convergence,
\[
\lim_{y\downarrow0}F(x+\ri y)=
\begin{cases}
\infty, &\text{if $x\ge0$}\\
F(x), &\text{if $x<0$.}
\end{cases}
\]
Thus, if $x>-c_\alpha$, then $\lim_{y\downarrow0}F(x+\ri
y)>F(-c_\alpha)=\frac{1}{\alpha}$, and there exists a unique $y_x$ in
$(0,\infty)$ 
such that $F(x+\ri y_x)=\frac{1}{\alpha}$. Thus, if we put
$v_\alpha(x)=y_x$, then $\alpha F(x+\ri y)<1$ and hence
$\im(H_\alpha(x+\ri y))>0$ for all $y$ in $(v_\alpha(x),\infty)$. Similarly
$\im(H_\alpha(x+\ri y))<0$ for $y$ in $(0,v_\alpha(x))$.

If $x\le -c_\alpha$, then for all $y$ in $(0,\infty)$ we have that
$F(x+\ri y)<F(x)\le F(-c_\alpha)=\frac{1}{\alpha}$, and hence that
$\im(H_\alpha(x+\ri y))>0$. Thus, if we put $v_\alpha(x)=0$ for $x$ in
$(-\infty,-c_\alpha]$, it follows
altogether that $v_\alpha$ satisfies \eqref{eq3.20}, and that
$v_\alpha$ is given by \eqref{eq3.6a}-\eqref{eq3.6}.

(iii) follows immediately from \eqref{eq3.21} in combination with
\eqref{eq3.6a}-\eqref{eq3.6}, and (iv) is a consequence of the way
$v_\alpha$ was defined in the proof of (ii).
\end{proofof}

In continuation of Lemma~\ref{intro v-alpha} we introduce next the
the following notation:
\begin{align}
\CG_\alpha&=\{x+\ri v_\alpha(x)\mid x\in\R\}
\label{eq3.4}
\\[.2cm]
\CG_\alpha'&=\{x+\ri v_\alpha(x)\mid
x\in[-c_\alpha,\infty)\}=\CG_\alpha\setminus(-\infty,-c_\alpha).
\label{eq3.4a}
\\[.2cm]
\CG_\alpha^+&=\{x+\ri y\mid x,y\in\R, \ \text{and} \ y>v_\alpha(x)\}.
\label{eq3.4b}
\end{align}
Note in particular that $0\notin\CG_\alpha$ (since $v_\alpha(0)>0$),
and that $\CG_\alpha,\CG_\alpha^+\subseteq\C\setminus[0,\infty)$.

\begin{proposition}\label{key formel} Let $\alpha$ be a positive
  number, and consider the free Gamma distribution $\nu_\alpha$ with
  parameter $\alpha$. Consider further the classical Gamma
  distribution $\mu_1$ with parameter 1 (cf.\ \eqref{eqI.4}). 
  We then have (cf.\ formulae \eqref{eqPL.6} and
  \eqref{eqPL.6a})

\begin{enumerate}[i]

\item $\CC_{\nu_\alpha}(\frac{1}{z})=\alpha G_{\mu_1}(z)$ for all $z$
  in $\C^+$.

\item $G_{\nu_\alpha}(H_\alpha(z))=\frac{1}{z}$ for all $z$ in
  $\CG_\alpha^+$.

\end{enumerate}
\end{proposition}

\begin{proof} (i) \ The classical Gamma distribution $\mu_\alpha$
has characteristic function 
\begin{equation}
\hat{\mu}_\alpha(u)=\exp\Big(\alpha\int_0^{\infty}\big(\e^{\ri
  ut}-1\big)\frac{\e^{-t}}{t}\6t\Big),  
\qquad(u\in\R),
\label{eq3.19a}
\end{equation}
(see e.g.\ \cite[Example~8.10]{Sa}).
By formula \eqref{eqPL.7} and Fubinis Theorem it follows then for
any $u$ in $(-\infty,0)$ that
\begin{equation*}
\begin{split}
\CC_{\nu_\alpha}(\ri u)&=\int_0^{\infty}\log\hat{\mu}_\alpha(ux)\e^{-x}\6x
=\alpha\int_0^{\infty}
\Big(\int_0^{\infty}(\e^{\ri uxt}-1)\frac{\e^{-t}}{t}\6t\Big)\e^{-x}\6x 
\\[.2cm]
&=\alpha\int_0^{\infty}\frac{\e^{-t}}{t}\Big(\frac{1}{1-\ri ut}-1\Big)\6t
=\alpha\ri u\int_0^{\infty}\frac{\e^{-t}}{1-\ri ut}\6t.
\end{split}
\end{equation*}
Setting $u=-\frac{1}{y}$, we find for any $y$ in $(0,\infty)$ that
\begin{equation*}
\CC_{\nu_\alpha}(\tfrac{1}{\ri y})
=\alpha\int_0^{\infty}\frac{\e^{-t}}{\ri y(1-\frac{t}{\ri y})}\6t
=\alpha\int_0^{\infty}\frac{\e^{-t}}{\ri y-t}\6t
=\alpha G_{\mu_1}(\ri y),
\end{equation*}
and by analytic continuation we conclude that
$\CC_{\nu_\alpha}(1/z)=\alpha G_{\mu_1}(z)$ for all $z$ in $\C^+$.

(ii) \ Recall from the definition of $\CC_{\nu_\alpha}$ (see
Subsection~\ref{PL free ID}) that
\[
\CC_{\nu_\alpha}(\tfrac{1}{z})
=\tfrac{1}{z}G_{\nu_\alpha}^{\<-1\>}(\tfrac{1}{z})-1
\]
for all $z$ in a region of the form $\Delta_{\eta,M}=\{z\in\C^+\mid
|\re(z)|<\eta\im(z), \ \im(z)>M\}$ for suitable positive numbers
$\eta$ and $M$. Taking (i) into account, we find that
\begin{equation}
G_{\nu_\alpha}^{\<-1\>}(\tfrac{1}{z})=z+z\alpha G_{\mu_1}(z)=H_\alpha(z),
\quad\text{and hence}\quad
\frac{1}{z}=G_{\nu_\alpha}(H_\alpha(z))
\label{eq3.19}
\end{equation}
for all $z$ in $\Delta_{\eta,M}$. Since $G_{\nu_\alpha}$ and
$H_{\alpha}$ are analytic on $\C^+$, it follows from
Lemma~\ref{intro v-alpha}(ii) and analytic continuation that the
latter equation in \eqref{eq3.19} holds for all $z$ in $\CG_\alpha^+$. 
This completes the proof.
\end{proof}

In order to combine Proposition~\ref{key formel}(ii) with the method
of Stieltjes inversion (and Lemma~\ref{BV key lemma}), we need some
further preparations, which  are presented in the series of lemmas to follow.

\begin{lemma}\label{H' ikke 0} 
For any positive number $\alpha$ we have that
\[
H_\alpha'(-c_\alpha)=0, \qand H'_\alpha(z)\ne0 \ \text{for all $z$ in
  $\CG_\alpha'\setminus\{-c_\alpha\}$}. 
\]
In fact,
\[
\re(H_\alpha'(x+\ri v_\alpha(x)))=-\alpha v_\alpha(x)
\frac{\partial}{\partial y}F(x+\ri v_\alpha(x))>0
\]
for all $x$ in $(-c_\alpha,\infty)$.
\end{lemma}

\begin{proof} Note first that by \eqref{eq3.9}-\eqref{eq3.10} we
 have that
\begin{equation*}
H_\alpha'(-c_\alpha)=1-\alpha\int_0^{\infty}\frac{t\e^{-t}}{(c_\alpha+t)^2}\6t
=1-\alpha F(-c_\alpha)=0. 
\end{equation*}
For $z=x+\ri y$ in $\C\setminus[0,\infty)$ we find next, by application of the
Cauchy-Riemann equations and \eqref{eq3.21}, that
\begin{equation*}
\begin{split}
\re(H'_\alpha(z))&=\frac{\partial}{\partial x}\re(H_\alpha(z))
=\frac{\partial}{\partial y}\im(H_\alpha(z))
\\[.2cm]
&=\frac{\partial}{\partial y}\big(y(1-\alpha F(z))
=(1-\alpha F(z))-\alpha y\frac{\partial}{\partial y}F(z).
\end{split}
\end{equation*}
For any $x$ in $(-c_\alpha,\infty)$ it follows thus from
\eqref{eq3.6} that 
\[
\re(H'_\alpha(x+\ri v_\alpha(x))=0
-\alpha v_\alpha(x)\frac{\partial}{\partial y}F(x+\ri v_\alpha(x)).
\]
The proof if concluded by noting that differentiation with respect to
$y$ in \eqref{eq3.0} leads to
\[
\frac{\partial}{\partial y}F(x+\ri y)
=-2y\int_0^{\infty}\frac{t\e^{-t}}{((x-t)^2+y^2)^2}\6t,
\]
where the right hand side is strictly negative whenever $y>0$.
\end{proof}

In the following lemma we collect some further properties of the
function $v_\alpha$, that will be needed in various parts of the
remainder of the paper. We defer the rather technical proof to
Appendix~\ref{tekniske beviser}.

\begin{lemma}\label{egensk v-alpha}
Let $\alpha$ be a positive number and consider the
  function $v_\alpha\colon\R\to[0,\infty)$ given by
      \eqref{eq3.6a}-\eqref{eq3.6}. Then $v_\alpha$ has the following
      properties: 

\begin{enumerate}[i]

\item $v_\alpha$ is continuous on $\R$ and analytic on
  $\R\setminus\{-c_\alpha\}$.

\item
  $\displaystyle{\lim_{x\to\infty}\frac{v_{\alpha}(x)}{x\e^{-x}}=\alpha\pi}$.

\item For any positive numbers $\delta,\gamma$ there
  exists a positive number $\alpha_0$ such that
\[
\sup_{x\in[\delta,\infty)}
\Big|\frac{v_\alpha(x)}{\alpha}-\pi x\e^{-x}\Big|\le\gamma,
\quad\text{whenever $\alpha\in(0,\alpha_0]$}.
\]

\end{enumerate}
\end{lemma}

\begin{lemma}\label{graense langs kurve}
Consider a fixed positive number $\alpha$. 

\begin{enumerate}[i]

\item For any $z$ in $(-\infty,-c_\alpha)$ we have that 
$z+\ri t\in\CG_\alpha^+$ for all positive $t$, and that
$H_\alpha(z+\ri t)\to H_\alpha(z)\in\R$
  non-tangentially (from $\C^+$ to $\R$), as $t\downarrow0$.

\item For any point $z$ in $\CG_\alpha'\setminus\{-c_\alpha\}$
there exists a vector $\gamma_z$ in $\C$
and a number $\epsilon_z$ in $(0,\infty)$, such that

\begin{enumerate}[a]

\item $z+t\gamma_z\in\CG_\alpha^+$ for all $t$ in
  $(0,\epsilon_z)$.

\item $H_\alpha(z+t\gamma_z)\to H_\alpha(z)\in\R$
  non-tangentially (from $\C^+$ to $\R$), as $t\downarrow0$.

\end{enumerate}
\end{enumerate}
\end{lemma}

\begin{proof} 

(i) \ Assume that $z\in(-\infty,-c_\alpha)$. According to
  Lemma~\ref{intro v-alpha}(ii) we have that $z+\ri t\in\CG_\alpha^+$
  and hence $H_\alpha(z+\ri t)\in\C^+$ for all
  positive $t$. It remains then to show that $\im(\frac{\d}{\d t}H_\alpha(z+\ri
  t))\ne0$ at $t=0$. Using \eqref{eq3.9} we find that
\begin{equation*}
\begin{split}
\frac{\d}{\d t}H_\alpha(z+\ri t)=\ri H_\alpha'(z+\ri t)
&=\ri\Big(1-\alpha\int_0^\infty\frac{s\e^{-s}}{(z+\ri t-s)^2}\6s\Big)
\\[.2cm]
&=\ri\Big(1-\alpha\int_0^\infty
\frac{(z-s-\ri t)^2s\e^{-s}}{((z-s)^2+t^2)^2}\6s\Big),
\end{split}
\end{equation*}
and hence at $t=0$ we have that
\[
\im\Big(\frac{\d}{\d t}H_\alpha(z+\ri t)\Big)
=1-\alpha\int_0^\infty
\frac{s\e^{-s}}{(z-s)^2}\6s>0,
\]
since $z<-c_\alpha$ (cf.\ \eqref{eq3.10}).

(ii) \ Assume that $z=x+\ri v_\alpha(x)$ for some $x$ in
$(-c_\alpha,\infty)$. We choose then a real number $r$, such that

\begin{enumerate}[1]

\item $r>v_\alpha'(x)$ (cf.\ Lemma~\ref{egensk v-alpha}(i)).

\item The vector $(1,r)$ is not perpendicular to
  the vector $(\im(H_\alpha'(z)),\re(H_\alpha'(z)))$ in $\R^2$
  (cf.\ Lemma~\ref{H' ikke 0}).

\end{enumerate}
We then put $\gamma_z=1+\ri r$. Condition (1) ensures that
we may choose a positive number $\epsilon_z$, such that claim (a) in
(ii) is satisfied. Indeed, otherwise we could choose a sequence $(t_n)$
of positive numbers, such that $t_n\to0$ as $n\to\infty$, and
$v_\alpha(x)+t_nr=\im(z+t_n\gamma_z)\le
v_\alpha(\re(z+t_n\gamma_z))=v_\alpha(x+t_n)$ for all $n$. This implies that
\[
r=\frac{v_\alpha(x)+t_nr-v_\alpha(x)}{t_n}
\le\frac{v_\alpha(x+t_n)-v_\alpha(x)}{t_n},
\]
for all $n$, which contradicts (1) and the fact that the right hand
side converges to $v_\alpha'(x)$ as $n\to\infty$.

Regarding assertion (b) in (ii), we remark first that
statement (ii) in Lemma~\ref{intro v-alpha} ensures that
$H_\alpha(z+t\gamma_z)\in\C^+$ for all $t$ in $(0,\epsilon_z)$.
We note next that
\begin{equation*}
\begin{split}
\im\Big(\diff H_\alpha(z+t\gamma_z)\Big)=\im\big(H_\alpha'(z)\gamma_z)=
\big\langle(\im(H_\alpha'(z)),\re(H_\alpha'(z))),(1,r)\big\rangle.
\end{split}
\end{equation*}
Condition (2) thus ensures that $\im(\frac{\d}{\d
  t}H_\alpha(z+t\gamma_z))\ne0$ at $t=0$, which implies (b).
This completes the proof.
\end{proof}

\begin{lemma}\label{non-tangential limit}
Let $\alpha$ be a strictly positive number, let $z$ be a
  point in $\CG_\alpha\setminus\{-c_\alpha\}$, and put $\xi=H_\alpha(z)=z+\alpha
  zG_{\mu_1}(z)\in\R$ (cf.\ Lemma~\ref{intro v-alpha}(iii)).

Then the Cauchy transform $G_{\nu_{\alpha}}$ of $\nu_{\alpha}$ has the
non-tangential limit $\frac{1}{z}$ at $\xi$. More precisely, for any
positive number $\delta$ we have that
\[
\lim_{w\to0\atop w\in\triangle_\delta}G_{\nu_\alpha}(\xi+w)=\frac{1}{z},
\]
where $\triangle_\delta$ is given by \eqref{eq1.1}.
\end{lemma}

\begin{proof} By Lemma~\ref{graense langs kurve} we may choose
  $\gamma_z$ in $\C$ and $\epsilon_z$ in $(0,\infty)$, such that
  $z+t\gamma_z\in\CG_\alpha^+$ for all $t$ in $(0,\epsilon_z)$, and such
  that $H_\alpha(z+t\gamma_z)\to H_\alpha(z)=\xi$ non-tangentially (from
  $\C^+$ to $\R$) as $t\downarrow0$. Using Proposition~\ref{key
    formel}(ii) it follows that 
\[
\lim_{t\downarrow0}G_{\nu_\alpha}(H_\alpha(z+t\gamma_z))
=\lim_{t\downarrow0}\frac{1}{z+t\gamma_z}=\frac{1}{z}
\]
(note that $z\ne0$, since $0\notin\CG_\alpha$).
Applying then Lemma~\ref{BV key lemma} (to the function $w\mapsto
-G_{\nu_\alpha}(\xi+w)$), we may conclude that actually
\[
\lim_{w\to0\atop
  w\in\triangle_\delta}G_{\nu_\alpha}(\xi+w)=\frac{1}{z}
\]
for any positive number $\delta$, as desired.
\end{proof}

For any $\alpha$ in $(0,\infty)$ we introduce next the function
$P_\alpha\colon\R\to\R$ (cf.\ Lemma~\ref{intro v-alpha}(iii))
given by
\begin{equation}
P_\alpha(x)=H_\alpha(x+\ri v_\alpha(x)), \qquad(x\in\R).
\label{eq3.13}
\end{equation}
In particular we put
\begin{equation}
s_\alpha=P_\alpha(-c_\alpha)=H_\alpha(-c_\alpha).
\label{eq3.13a}
\end{equation}
In the following lemma we collect some properties of $P_\alpha$ that
will be needed in the sequel. We defer the rather technical proof
to Appendix~\ref{tekniske beviser}.

\begin{lemma}\label{egensk P-alpha}
For any positive number $\alpha$ the function
$P_\alpha\colon\R\to\R$ has the following properties:

\begin{enumerate}[i]

\item $P_\alpha$ is continuous on $\R$ and analytic on
  $\R\setminus\{-c_\alpha\}$.

\item $P_\alpha$ satisfies that
\begin{equation}
P_\alpha(x)=
\begin{cases}
x+\alpha+\alpha\int_0^\infty\frac{t\e^{-t}}{x-t}\6t, &\text{if
  $x<-c_\alpha$},
\\
2x+\alpha-\alpha\int_0^\infty\frac{t^2\e^{-t}}{(x-t)^2+v_\alpha(x)^2}\6t,
&\text{if $x\ge-c_\alpha$}.
\end{cases}
\label{eq3.14}
\end{equation}

\item The number $s_\alpha:=P_\alpha(-c_\alpha)$ is strictly positive.

\item $\displaystyle{
\lim_{x\to\infty}\big(x+\alpha-P_\alpha(x)\big)=0}$.

\item $P_\alpha$ is a strictly increasing bijection of $\R$ onto $\R$,
  and $P_\alpha'(x)>0$ for all $x$ in $\R\setminus\{-c_\alpha\}$.

\end{enumerate}
\end{lemma}

\bigskip
\noindent
\[
\includegraphics[height=8.0cm]{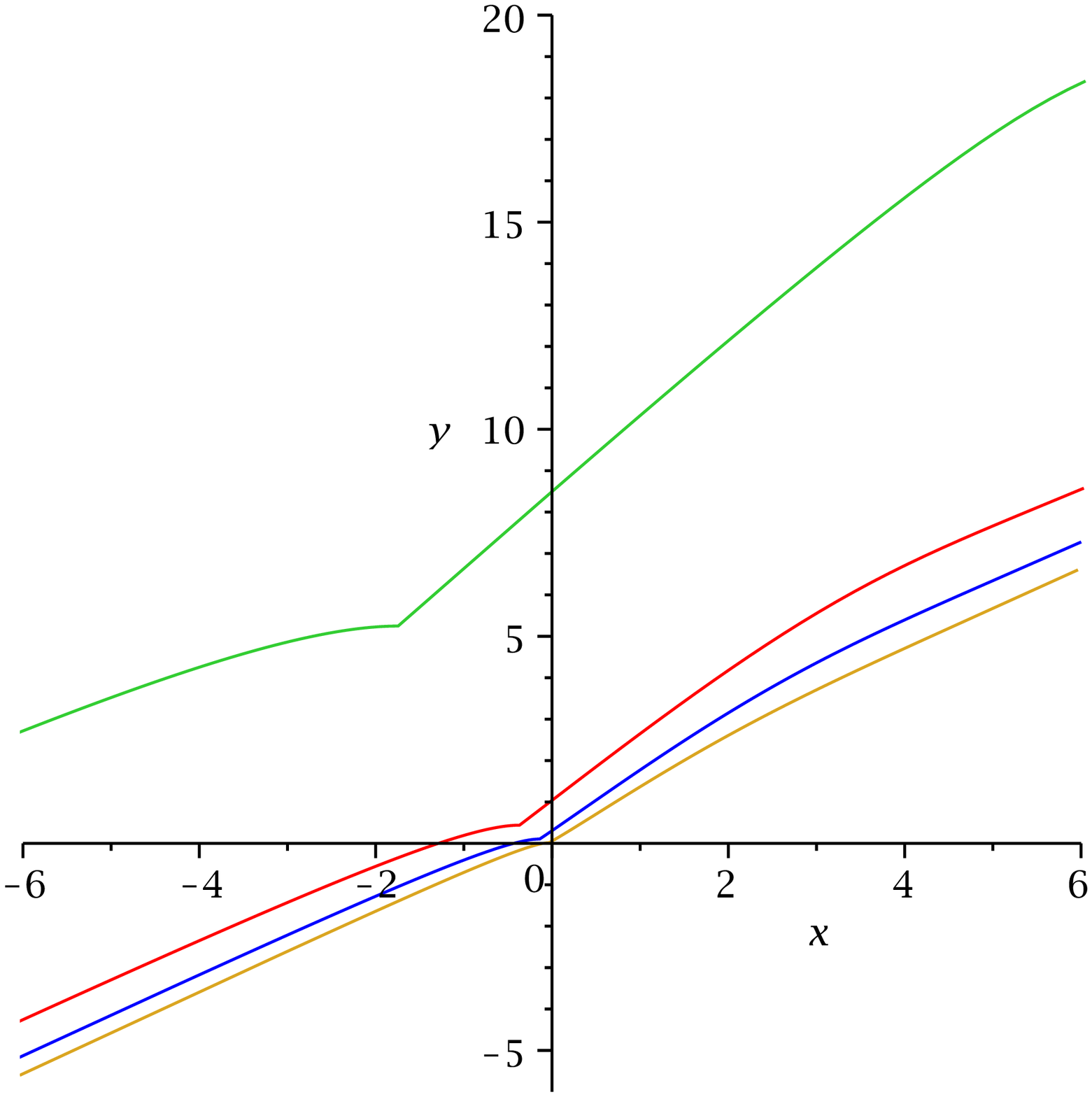}
\]
\begin{center}
{\footnotesize The graphs of the functions $P_{1/2}$, $P_1$, $P_2$ and
  $P_{10}$.}
\end{center}

\begin{theorem}\label{thm abs kont}
Let $\alpha$ be a positive number, and consider the
  function $Q_\alpha\colon\R\to[0,\infty)$
defined by 
\begin{equation}
Q_\alpha(x)=\frac{v_\alpha(x)}{x^2+v_\alpha(x)^2}, \qquad(x\in\R).
\label{eq3.1}
\end{equation}
We then have

\begin{enumerate}[i]

\item The free Gamma distribution $\nu_\alpha$ with parameter $\alpha$
  is absolutely continuous (with respect to Lebesgue measure).

\item The density $f_\alpha$ of $\nu_\alpha$ is given by
\begin{equation}
f_\alpha(\xi)=
\begin{cases}
0, &\text{if $\xi\in(-\infty,s_\alpha]$},\\
\frac{1}{\pi}Q_\alpha(P_\alpha^{\brinv}(\xi)), &\text{if
  $\xi\in(s_\alpha,\infty)$},
\end{cases}
\label{eq3.15}
\end{equation}
where $P_\alpha^{\brinv}$ denotes the inverse of $P_\alpha$ (cf.\ (v)
in Lemma~\ref{egensk P-alpha}).

\item The support of $\nu_\alpha$ is $[s_\alpha,\infty)$.

\item $f_\alpha$ is analytic on $(s_\alpha,\infty)$.

\end{enumerate}
\end{theorem}

\bigskip
\noindent
\begin{minipage}[c]{0.5\textwidth}
\[
\includegraphics[height=8.0cm]{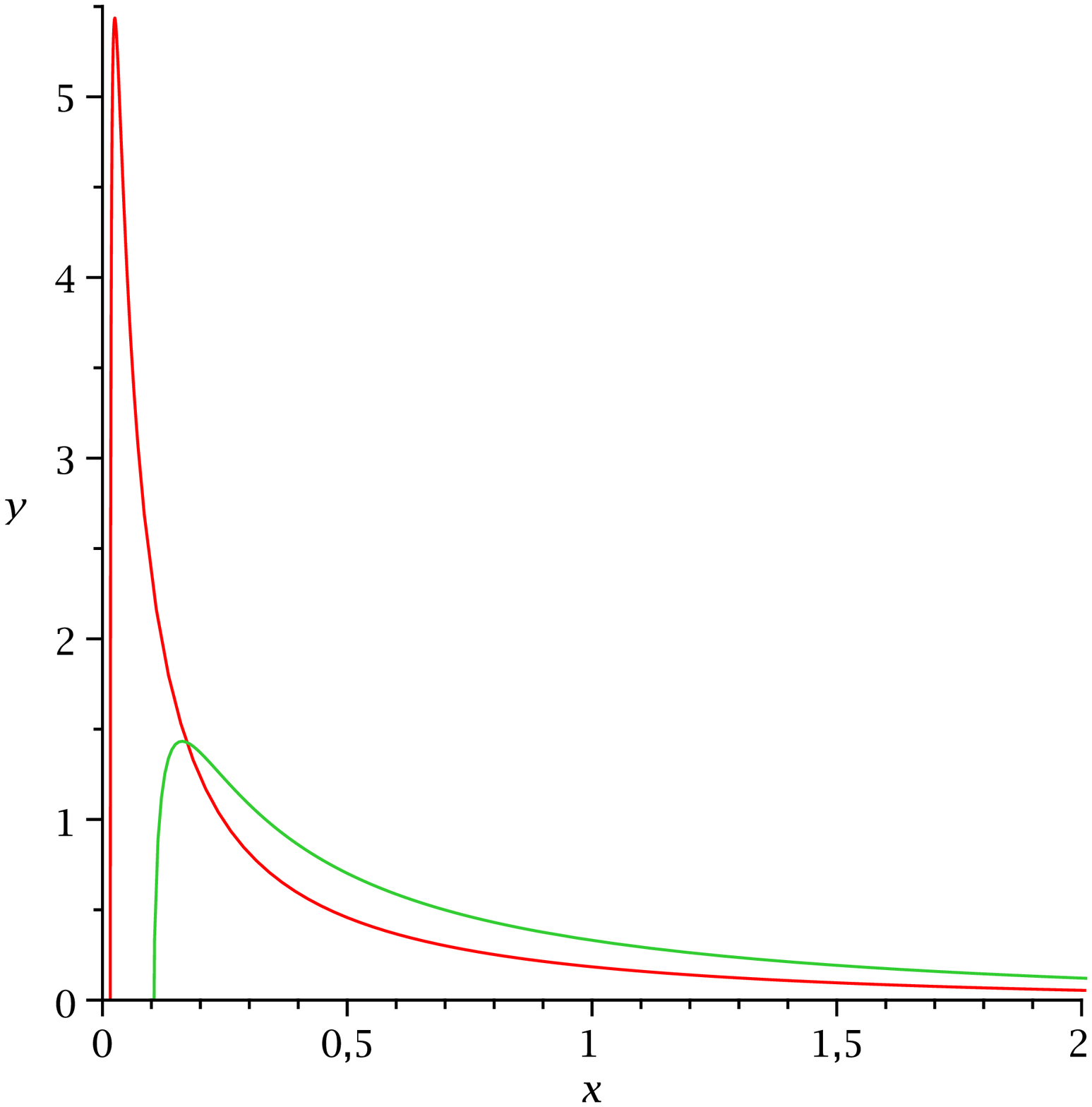}
\]
\end{minipage}
\hfill
\begin{minipage}[c]{0.5\textwidth}
\[
\includegraphics[height=8.0cm]{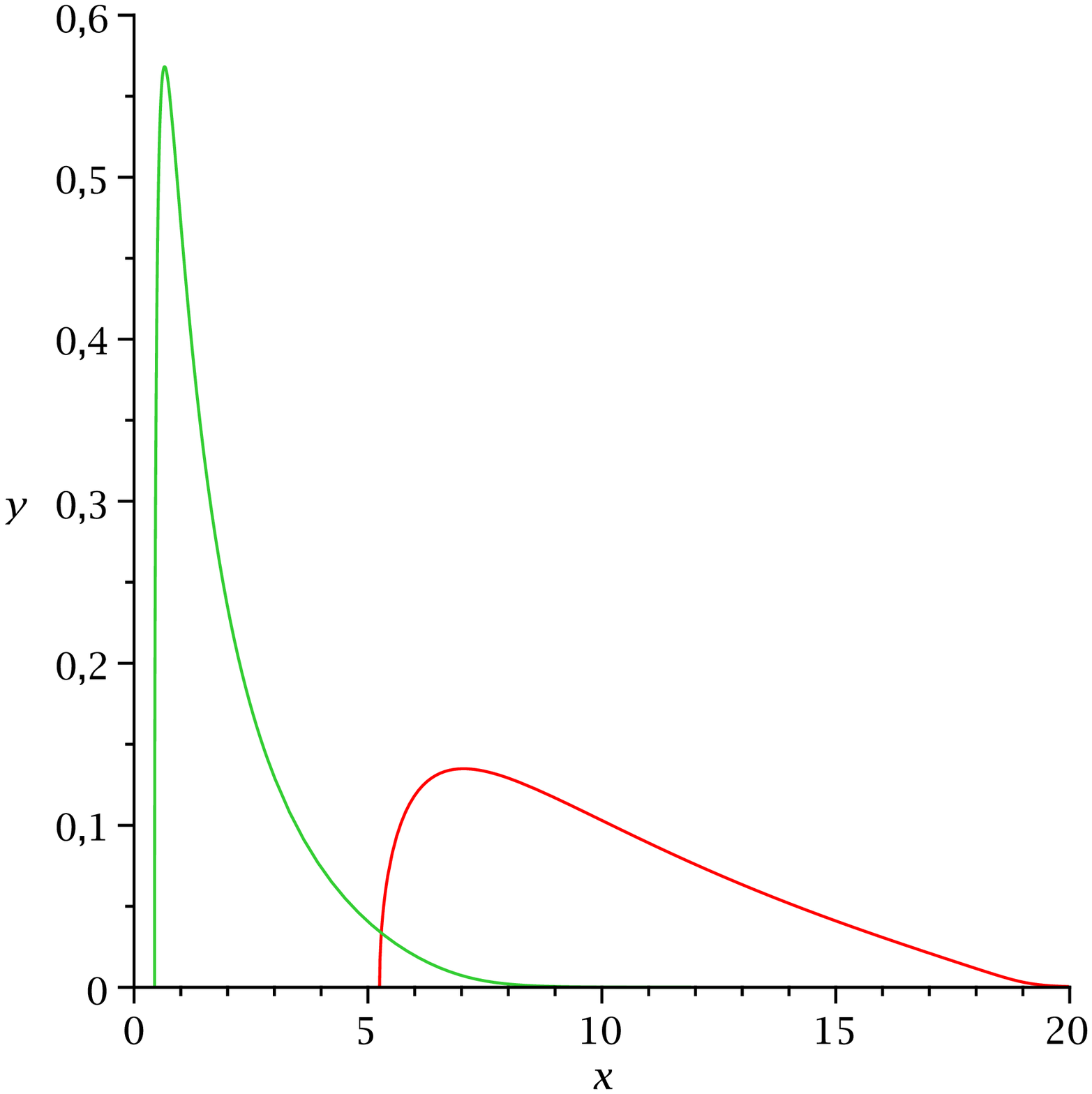}
\]
\end{minipage}
\begin{center}
{\footnotesize The graphs of the densities of the free gamma
  distributions $\nu_{1/2},\nu_1$ and,
  respectively, $\nu_2,\nu_{10}$.}
\end{center}

\begin{proofof}[Proof of Theorem~\ref{thm abs kont}.] 
(i) \ As described in Subsection~\ref{PL Stieltjes
    inversion}, the singular part of $\nu_\alpha$ (with respect to
  Lebesgue measure) is concentrated on the set
\begin{equation*}
S=\big\{\xi\in\R\bigm| 
\textstyle{\lim_{y\downarrow0}|G_{\nu_\alpha}(\xi+\ri y)|
=\infty\big\}}.
\end{equation*}
From Lemma~\ref{egensk P-alpha} it follows that $P_\alpha$ is a
continuous, increasing bijection of $\R$ onto $\R$. By \eqref{eq3.13}
this implies that $H_\alpha$ maps $\CG_\alpha$ bijectively onto
$\R$. For any $\xi$ in $\R\setminus\{s_{\alpha}\}$ it follows then
by Lemma~\ref{non-tangential limit} that
\begin{equation}
\lim_{y\downarrow0}G_{\nu_\alpha}(\xi+\ri y)=\frac{1}{z},
\label{eq3.17}
\end{equation}
where $z$ is the unique point on $\CG_\alpha\setminus\{-c_\alpha\}$,
such that $H_\alpha(z)=\xi$. We may therefore conclude that
$S\subseteq\{s_\alpha\}$, and the proof of (i) is completed, if
we verify that $\nu_\alpha$ has no atom at $s_\alpha$. As
mentioned in Subsection~\ref{PL free ID} it follows from (the proof of)
\cite[Proposition~5.12]{BV} that $\nu_\alpha$ has at most one atom,
which, if it exists, is necessarily equal to the non-tangential limit of
the Voiculescu transform $\phi_{\nu_\alpha}(z)$ as $z\to0$, $z\in\C^+$.
Note here for $z$ in $\C^+$ that by \eqref{eqPL.6b} and 
Proposition~\ref{key formel}(i),
\[
\phi_{\nu_\alpha}(z)=z\CC_{\nu_\alpha}(\tfrac{1}{z})=\alpha
zG_{\mu_1}(z)=\alpha+\alpha\int_0^{\infty}\frac{t\e^{-t}}{z-t}\6t
\]
(cf.\ formula \eqref{eq3.12}).
Hence, by dominated convergence,
\[
\lim_{y\downarrow0}\phi_{\nu_\alpha}(\ri y)
=\lim_{y\downarrow0}\Big(\alpha
+\alpha\int_0^{\infty}\frac{t\e^{-t}}{\ri y-t}\6t\Big)
=\alpha-\alpha=0.
\]
It follows that the only possible atom for $\nu_\alpha$ is 0, and since
$s_\alpha>0$ (cf.\ Lemma~\ref{egensk P-alpha}(ii)), we
conclude that $\nu_\alpha$ has no singular part.

(ii) \ By Stieltjes inversion 
(cf.\ Subsection~\ref{PL Stieltjes inversion}), the formula 
\[
f_{\alpha}(\xi)=-\frac{1}{\pi}\lim_{y\downarrow0}\im(G_{\nu_\alpha}(\xi+\ri
y)),
\]
produces an almost everywhere defined density for $\nu_\alpha$ with
respect to Lebesgue 
measure. According to \eqref{eq3.17} we have for all $\xi$ in
$\R\setminus\{s_\alpha\}$ that
\[
f_{\alpha}(\xi)=-\frac{1}{\pi}\im\Big(\frac{1}{z}\Big),
\]
where $z$ is the unique point in $\CG_\alpha\setminus\{-c_\alpha\}$, such
that $H_\alpha(z)=\xi$. Writing $z=x+\ri v_\alpha(x)$ for some
(unique) $x$ in $\R\setminus\{-c_{\alpha}\}$, we have that
\[
\xi=H_\alpha(x+\ri v_\alpha(x))=P_\alpha(x), \quad\text{so that}\quad
x=P_\alpha^{\brinv}(\xi),
\]
and therefore
\[
f_{\alpha}(\xi)
=-\frac{1}{\pi}\im\Big(\frac{1}{x+\ri v_\alpha(x)}\Big)
=\frac{1}{\pi}\frac{v_\alpha(x)}{x^2+v_\alpha(x)^2}
=\frac{1}{\pi}Q_\alpha(x)=\frac{1}{\pi}Q_\alpha(P_\alpha^{\brinv}(\xi)).
\]
The proof of (ii) is completed by noting that if
$\xi<s_\alpha$, then $x<-c_{\alpha}$, so that $v_\alpha(x)=0$, and
therefore $f_\alpha(\xi)=0$ by the previous calculation.

(iii) \ This is an immediate consequence of (i), (ii) and the fact
that $v_\alpha$ (and hence $Q_\alpha$) is strictly positive on
$(-c_\alpha,\infty)$ (cf.\ Lemma~\ref{intro v-alpha}(iv)).

(iv) \ Since $v_\alpha$ is analytic on $(-c_\alpha,\infty)$
(cf.\ Lemma~\ref{egensk v-alpha}(i)), it follows immediately from
\eqref{eq3.1} that so is $Q_\alpha$. By (i) and (v) of Lemma~\ref{egensk
  P-alpha} the function $P_\alpha$ is analytic on $(-c_\alpha,\infty)$
with strictly positive derivative. This implies that $P^\brinv_\alpha$
is analytic on $(s_\alpha,\infty)$, and altogether we thus
conclude that 
$f_\alpha=\tfrac{1}{\pi}Q_\alpha\circ P_\alpha^\brinv$ is analytic on
$(s_\alpha,\infty)$. 
\end{proofof}

\section{Behavior at the limits of the support}\label{sec asympt opf}

In this section we study the behavior of the density $f_\alpha$ of
$\nu_\alpha$ around the lower bound $s_\alpha$ of the support
and at infinity. We start with the latter aspect.

\begin{proposition}\label{asympt opfoersel af taethed}
Let $\alpha$ be a positive number, and consider 
the density $f_\alpha$ of $\nu_\alpha$ (cf.\ \eqref{eq3.15}). We then have
\[
\lim_{\xi\to\infty}\frac{f_{\alpha}(\xi)}{\xi^{-1}\e^{-\xi}}
=\alpha\e^{\alpha}.
\]
\end{proposition}

\begin{proof} Consider the function $P_\alpha$ introduced in \eqref{eq3.13}.
Since $P_\alpha$ is a strictly increasing bijection of $\R$ onto $\R$,
it suffices to prove that 
\[
\lim_{x\to\infty}f_{\alpha}(P_\alpha(x))P_\alpha(x)\e^{P_\alpha(x)}=\alpha\e^{\alpha}.
\]
Using Theorem~\ref{thm abs kont}(ii), Lemma~\ref{egensk v-alpha}(iv)
and Lemma~\ref{egensk P-alpha}(iv), we find that
\begin{equation*}
\begin{split}
\lim_{x\to\infty}f_{\alpha}(P_\alpha(x))P_\alpha(x)\e^{P_\alpha(x)}
&=
\lim_{x\to\infty}\frac{1}{\pi}\frac{v_\alpha(x)}{x^2+v_\alpha(x)^2}
P_\alpha(x)\e^{P_\alpha(x)}
\\[.2cm]&=
\lim_{x\to\infty}\frac{1}{\pi}\Big(\frac{x^2}{x^2+v_\alpha(x)^2}\Big)
\Big(\frac{v_{\alpha}(x)}{x\e^{-x}}\Big)
\Big(\frac{P_\alpha(x)}{x}\Big)\e^{-x+P_\alpha(x)}
\\[.2cm]&=
\frac{1}{\pi}\cdot1\cdot\alpha\pi\cdot1\cdot\e^{\alpha}
=
\alpha\e^{\alpha},
\end{split}
\end{equation*}
as desired.
\end{proof}

We turn next to the behavior of $f_\alpha(\xi)$ as $\xi\downarrow
s_\alpha$. We study initially how $s_\alpha$ varies as a function of
$\alpha$.

\begin{proposition}\label{egensk s-alpha}
For any positive number $\alpha$ consider the
  function $H_\alpha$ and the quantities $c_\alpha$ and $s_\alpha$
  defined by \eqref{eq3.12},\eqref{eq3.10} and \eqref{eq3.13a},
  respectively. We then have 

\begin{enumerate}[i]

\item $H_\alpha$ satisfies the differential equation:
\[
H'_\alpha(z)=
\alpha+z+(z^{-1}-1)H_\alpha(z),\qquad(z\in\C\setminus[0,\infty)).
\]

\item $s_\alpha=\frac{c_\alpha}{1+c_\alpha}(\alpha-c_\alpha)$.

\item $H_\alpha''(-c_\alpha)=1-\frac{s_\alpha}{c_{\alpha}^2}<0$.

\item $\lim_{\alpha\to0}s_\alpha=0$, and
  $\lim_{\alpha\to\infty}s_\alpha=\infty$. 

\item $c_\alpha$ is an analytic function of $\alpha$, and
$\frac{\d c_\alpha}{\d\alpha}
=\frac{c_\alpha(1+c_\alpha)}{\alpha(\alpha-2c_\alpha-c_\alpha^2)}$.

\item $s_\alpha$ is an analytic function of $\alpha$, and
$\frac{\d s_\alpha}{\d\alpha}
=\frac{c_\alpha(\alpha+1)}{\alpha(1+c_\alpha)}$. In particular
$s_\alpha$ is a strictly increasing function of $\alpha$.

\end{enumerate}
\end{proposition}

\begin{proof} (i) \ Differentiation in the first equality in
  \eqref{eq3.12} and partial integration leads to
\begin{equation*}
\begin{split}
H_\alpha'(z)&=1+\alpha\int_0^{\infty}\frac{\e^{-t}}{z-t}\6t
-\alpha z\int_0^\infty\frac{\e^{-t}}{(z-t)^2}\6t
\\[.2cm]
&=z^{-1}H_\alpha(z)-\alpha z
\Big(\Big[\frac{\e^{-t}}{z-t}\Big]_0^{\infty}
+\int_0^\infty\frac{\e^{-t}}{z-t}\6t\Big)
\\[.2cm]
&=z^{-1}H_\alpha(z)+\alpha-(H_\alpha(z)-z)
=\alpha+z+(z^{-1}-1)H_\alpha(z)
\end{split}
\end{equation*}
for all $z$ in $\C\setminus[0,\infty)$.

(ii) \ From Lemma~\ref{H' ikke 0}, (i) and \eqref{eq3.13a} it follows that
\[
0=H_{\alpha}'(-c_\alpha)=\alpha-c_\alpha+(-c_{\alpha}^{-1}-1)
H_\alpha(-c_{\alpha}) 
=\alpha-c_\alpha-\tfrac{1+c_\alpha}{c_\alpha}s_\alpha,
\]
from which (i) follows immediately.

(iii) \ Differentiation in (i) leads to the
formula:
\[
H_\alpha''(z)=1-z^{-2}H_\alpha(z)+(z^{-1}-1)H'_\alpha(z),
\qquad(z\in\C\setminus[0,\infty)).
\]
Combining this with Lemma~\ref{H' ikke 0}, we find that
\[
H_\alpha''(-c_\alpha)=1-c_\alpha^{-2}H_\alpha(-c_\alpha)+0
=1-c_\alpha^{-2}s_\alpha.
\]
At the same time it follows from \eqref{eq3.9} that
\[
H_\alpha''(-c_\alpha)=2\alpha\int_0^{\infty}\frac{t\e^{-t}}{(-c_\alpha-t)^3}\6t
=-2\alpha\int_0^{\infty}\frac{t\e^{-t}}{(c_\alpha+t)^3}\6t<0,
\]
and thus (iii) is established.

(iv) \ From Lemma~\ref{intro v-alpha}(i) we know that $c_\alpha\to0$
as $\alpha\to0$, and that $c_\alpha\to\infty$ as $\alpha\to\infty$. In
combination with (ii) and (iii), respectively, it follows that
$s_\alpha$ has the same properties.

(v) \  For any $x$ in $(0,\infty)$ it follows from
\eqref{eq3.9}-\eqref{eq3.0} that
\begin{equation*}
H_\alpha'(-x)=1-\alpha\int_0^{\infty}\frac{t\e^{-t}}{(x+t)^2}\6t
=1-\alpha F(-x), 
\end{equation*}
so that
\begin{equation}
F(-x)=\frac{1}{\alpha}(1-H'_\alpha(-x)), \qand
F'(-x)=-\frac{1}{\alpha}H''_{\alpha}(-x), \qquad(x\in(0,\infty)).
\label{eq4.7}
\end{equation}
Using (iii) and (ii) we find thus that
\begin{equation*}
F'(-c_\alpha)=-\frac{1}{\alpha}H''_{\alpha}(-c_\alpha)
=\frac{1}{\alpha}(c_\alpha^{-2}s_\alpha-1)
=\frac{1}{\alpha}
\Big(\frac{\alpha-c_{\alpha}}{c_{\alpha}(1+c_\alpha)}-1\Big)
=\frac{\alpha-2c_\alpha-c_\alpha^2}{\alpha c_\alpha(1+c_\alpha)}.
\end{equation*}
In particular we see from (iii) that $F'(-c_\alpha)>0$, and \eqref{eq4.7}
shows that $F$ is analytic on $(-\infty,0)$. From the defining formula:
$F(-c_\alpha)=\frac{1}{\alpha}$
and the implicit function theorem (for analytic functions; see
\cite[Theorem~7.6]{FG}) it follows thus 
that $c_\alpha$ is an analytic function of $\alpha$ with
derivative given by
\[
\frac{\d c_\alpha}{\d\alpha}
=\frac{1}{\alpha^2F'(-c_\alpha)} 
=\frac{c_\alpha(1+c_\alpha)}{\alpha(\alpha-2c_\alpha-c_\alpha^2)},
\qquad(\alpha\in(0,\infty)).
\]

(vi) \ From (ii) and (v) it is clear that $s_\alpha$ is an analytic
function of $\alpha$. Consider now the function
$S\colon(0,\infty)\times(0,\infty)\to\R$ given by
\[
S(c,\alpha)=\frac{c}{1+c}(\alpha-c), \qquad(c,\alpha\in(0,\infty)),
\]
and note that $s_\alpha=S(c_\alpha,\alpha)$ for all positive
$\alpha$, and that
\[
\frac{\partial S}{\partial c}(c,\alpha)
=\frac{\alpha-2c-c^2}{(1+c)^2}, \qand
\frac{\partial S}{\partial\alpha}(c,\alpha)=\frac{c}{1+c},
\qquad (c,\alpha\in(0,\infty)).
\]
Using the chain rule and (v) it follows thus that
\begin{equation*}
\begin{split}
\frac{\d s_\alpha}{\d\alpha}
&=\frac{\partial S}{\partial c}(c_\alpha,\alpha)\frac{\d c_\alpha}{\d\alpha}
+\frac{\partial S}{\partial\alpha}(c_\alpha,\alpha)
=\frac{\alpha-2c_\alpha-c_\alpha^2}{(1+c_\alpha)^2}\cdot
\frac{c_\alpha(1+c_\alpha)}{\alpha(\alpha-2c_\alpha-c_\alpha^2)}
+\frac{c_\alpha}{1+c_\alpha}
\\[.2cm]
&=\frac{c_\alpha}{\alpha(1+c_\alpha)}+\frac{c_\alpha}{1+c_\alpha}
=\frac{c_\alpha(\alpha+1)}{\alpha(1+c_\alpha)},
\end{split}
\end{equation*}
and this completes the proof.
\end{proof}

Let $a$ be a real number contained in an interval $I$.
For functions $g,h\colon I\to\C$, such that $0\notin h(I)$,
we use in the following proposition the notation: ``$g(x)\sim h(x)$ as
$x\to a$'' to express that $\lim_{x\to a}\frac{g(x)}{h(x)}=1$.

\begin{proposition}
For any positive number $\alpha$ we put
\[
\gamma_\alpha=\frac{6H_\alpha''(-c_\alpha)}{H_\alpha'''(-c_\alpha)}.
\]
Then $\gamma_\alpha>0$, and we have that

\begin{enumerate}[i]

\item $v_\alpha(x)\sim \gamma_\alpha^{1/2}(x+c_\alpha)^{1/2}$, as
  $x\downarrow -c_\alpha$.

\item $\displaystyle{\lim_{x\downarrow
    -c_\alpha}\frac{P_\alpha(x)-s_\alpha}{x+c_\alpha}
=-\tfrac{1}{2}\gamma_\alpha H_\alpha''(-c_\alpha)>0}$.

\item $f_\alpha(\xi)\sim\frac{\sqrt{2}}
{\pi c_\alpha\sqrt{s_\alpha-c_\alpha^2}} 
(\xi-s_\alpha)^{1/2}$, as $\xi\downarrow s_\alpha$.

\end{enumerate}
\end{proposition}

\begin{proof} Using formula \eqref{eq3.9} we note first for any $k$ in
  $\{2,3,4,\ldots\}$ that
\begin{equation}
H_\alpha^{(k)}(-c_\alpha)
=-\alpha k!\int_0^{\infty}\frac{t\e^{-t}}{(t+c_\alpha)^{k+1}}\6t<0,
\label{eq4.7a}
\end{equation}
and in particular this verifies that $\gamma_\alpha>0$. 

(i) \ Using formula \eqref{eq4.6} (in Appendix~\ref{tekniske beviser})
we find that
\[
\frac{\d}{\d x}\big(v_\alpha(x)^2\big)=2v_\alpha(x)v_\alpha'(x)
\underset{x\downarrow -c_\alpha}{\longrightarrow}
\frac{2\int_0^{\infty}\frac{t\e^{-t}}{(t+c_\alpha)^3}\6t}
{\int_0^{\infty}\frac{t\e^{-t}}{(t+c_\alpha)^4}\6t}=\gamma_\alpha,
\]
where the last equality results from \eqref{eq4.7a}. Since
$v_\alpha(-c_\alpha)=0$, it follows from the above calculation and the
mean value theorem that $v_\alpha(x)^2\sim\gamma_\alpha(x+c_\alpha)$
as $x\downarrow -c_\alpha$, and this proves (i).

(ii) \ Using that $H_{\alpha}(-c_\alpha)=s_\alpha$ and
$H_\alpha(-c_\alpha)=0$ (cf.\ Lemma~\ref{H' ikke 0}) we find by Taylor
expansion that
\[
P_{\alpha}(x)=\re\big(H_\alpha(x+\ri v_\alpha(x))\big)
=\re\Big(s_\alpha+\tfrac{1}{2}H_\alpha''(-c_\alpha)
\big(x+c_\alpha+\ri v_\alpha(x)\big)^2+
o\big(|x+c_\alpha+\ri v_\alpha(x)|^2\big)\Big)
\]
and hence by application of (i),
\[
\frac{P_\alpha(x)-s_\alpha}{x+c_\alpha}
=\frac{1}{2}H_\alpha''(-c_\alpha)\Big(x+c_\alpha
-\frac{v_\alpha(x)^2}{x+c_\alpha}\Big)
+\frac{o((x+c_\alpha)^2+v_\alpha(x)^2)}{x+c_\alpha}
\longrightarrow
-\tfrac{1}{2}H_\alpha''(-c_\alpha)\gamma_\alpha,
\]
as $x\downarrow -c_\alpha$. Since $\gamma_\alpha>0$, formula
\eqref{eq4.7a} shows that the resulting expression above is positive,
and hence (ii) is established.

(iii) \ Recall from Theorem~\ref{thm abs kont} that
\[
f_\alpha(P_\alpha(x))=\frac{1}{\pi}\frac{v_\alpha(x)}{x^2+v_\alpha(x)^2},
\qquad(x\in[-c_\alpha,\infty)).
\]
By application of (i) it follows thus that
\begin{equation}
f_\alpha(P_\alpha(x))\sim\frac{\gamma_\alpha^{1/2}}{\pi
  c_\alpha^2}(x+c_\alpha)^{1/2}, \quad\text{as $x\downarrow
  -c_\alpha$}.
\label{eq4.7b}
\end{equation}
Using (ii) we have also that
\[
\lim_{\xi\downarrow s_\alpha}
\Big(\frac{\xi-s_\alpha}{P_\alpha^\brinv(\xi)+c_\alpha}\Big)
=\Big(\frac{P_\alpha(P_\alpha^\brinv(\xi))-s_\alpha}
{P_\alpha^\brinv(\xi)+c_\alpha}\Big)
=-\tfrac{1}{2}\gamma_\alpha H_\alpha''(-c_\alpha),
\]
and hence
\[
P_\alpha^\brinv(\xi)+c_\alpha\sim 
-\frac{2}{\gamma_\alpha H_\alpha''(-c_\alpha)}(\xi-s_\alpha), 
\quad\text{as $\xi\downarrow -s_\alpha$.}
\]
Combining this with \eqref{eq4.7b} we find that
\[
f_\alpha(\xi)
\sim\frac{\gamma_\alpha^{1/2}(P_\alpha^\brinv(\xi)+c_\alpha)^{1/2}}
{\pi c_\alpha^2}
\sim\frac{\sqrt{2}}{\pi c_\alpha^2(-H_\alpha''(-c_\alpha))^{1/2}}
(\xi-s_\alpha)^{1/2},
\]
as $\xi\downarrow s_\alpha$. Applying finally Proposition~\ref{egensk
  s-alpha}(iii), we obtain (iii).
\end{proof}

\section{Unimodality}\label{sec unimodal}

In this section we establish unimodality of the densities
$f_\alpha$. We start with a few preparatory results.

\begin{lemma}\label{key lemma}
For each positive number $R$, let $\Phi_R\colon(0,\pi)\to(0,\infty)$ be
the function given by
\[
\Phi_R(\theta)=F(R\sin(\theta)\e^{\ri\theta}), \qquad(\theta\in(0,\pi)),
\]
where $F$ is the function introduced in \eqref{eq3.0}.
Then for any $R$ in $(0,\infty)$ there exists a unique number
$\theta_R$ in $(0,\pi)$ such that $\Phi_R$ is strictly decreasing on
$(0,\theta_R]$ and strictly increasing on $[\theta_R,\pi)$.
\end{lemma}

\begin{proof} We note first that for any $r$ in $(0,\infty)$ and
  $\theta$ in $(-\pi,\pi]$ we have, using the change of variables
  $t=ru$, that
\begin{equation*}
\begin{split}
F(r\e^{\ri\theta})&=\int_0^{\infty}\frac{t\e^{-t}}{(r\cos(\theta)-t)^2
+r^2\sin^2(\theta)}\6t
=\int_0^{\infty}\frac{ru\e^{-ru}}{r^2(\cos(\theta)-u)^2
+r^2\sin^2(\theta)}r\6u
\\[.2cm]
&=\int_0^{\infty}\frac{u\e^{-ru}}{1-2u\cos(\theta)+u^2}\6u.
\end{split}
\end{equation*}
Hence, for a fixed positive number $R$ we have that
\[
\Phi_R(\theta)=F(R\sin(\theta)\e^{\ri\theta})=
\int_0^\infty\frac{u\e^{-uR\sin(\theta)}}{1-2u\cos(\theta)+u^2}\6u,
\qquad(\theta\in(0,\pi)).
\]
Then define the function $\Psi_R\colon(-1,1)\to(0,\infty)$ by
\[
\Psi_R(s)=\int_0^\infty\frac{u\e^{-uR\sqrt{1-s^2}}}{1-2us+u^2}\6u,
\qquad(s\in(-1,1)),
\]
so that $\Phi_R(\theta)=\Psi_R(\cos(\theta))$ for $\theta$ in
$(0,\pi)$. 
Since the function $\theta\mapsto\cos(\theta)$ is strictly decreasing
on $(0,\pi)$, it suffices then to show that $\Psi_R$ is strictly
decreasing on $(-1,\eta_R]$ and strictly increasing on $[\eta_R,1)$
for some number $\eta_R$ in $(-1,1)$. For this we consider for any
$u$ in $(0,\infty)$ the function
$\psi_{R,u}\colon(-1,1)\to(0,\infty)$ given by
\[
\psi_{R,u}(s)=\frac{u\e^{-uR\sqrt{1-s^2}}}{1-2us+u^2},
\qquad(s\in(-1,1)).
\]
By a standard application of the theorem on differentiation under the
integral sign, it follows that $\Psi_R$ is differentiable on $(-1,1)$
with derivative
\begin{equation}
\Psi_{R}'(s)=\int_0^{\infty}\frac{\d}{\d s}\psi_{R,u}(s)\6u,
\qquad(s\in(-1,1)).
\label{eq3.3a}
\end{equation}
For any $u$ in $(0,\infty)$ and $s$ in $(-1,1)$ we note further that
\begin{equation*}
\begin{split}
\tfrac{\d}{\d s}\ln(\psi_{R,u}(s))
&=\tfrac{\d}{\d s}\big(\ln(u)-uR\sqrt{1-s^2}-\ln(1-2us+u^2)\big)
\\[.2cm]
&=uRs(1-s^2)^{-1/2}+2u(1-2us+u^2)^{-1},
\end{split}
\end{equation*}
so that
\begin{equation*}
\begin{split}
\tfrac{\d^2}{\d s^2}\ln(\psi_{R,u}(s))
&=uR(1-s^2)^{-1/2}+uRs^2(1-s^2)^{-3/2}+4u^2(1-2us+u^2)^{-2}
\\[.2cm]
&=uR(1-s^2)^{-3/2}+4u^2(1-2us+u^2)^{-2}>0.
\end{split}
\end{equation*}
Since
\[
\frac{\d^2}{\d s^2}\ln(\psi_{R,u}(s))
=\frac{\psi_{R,u}''(s)}{\psi_{R,u}(s)}
-\frac{\psi_{R,u}'(s)^2}{\psi_{R,u}(s)^2},
\]
we may thus conclude that $\psi_{R,u}''>0$ and hence that
$\psi_{R,u}'$ is strictly increasing on $(-1,1)$. Since this holds for
all $u$ in $(0,\infty)$, it follows further from \eqref{eq3.3a} that
$\Psi_R'$ is \emph{strictly} increasing on $(-1,1)$. Thus, $\Psi_R$ is
either strictly increasing, strictly decreasing or of the form
asserted above. However, by Fatou's Lemma,
\[
\liminf_{s\uparrow 1}\Psi_R(s)\ge\int_0^\infty\frac{u}{(1-u)^2}\6u=\infty,
\qand 
\liminf_{s\downarrow -1}\Psi_R(s)\ge\int_0^{\infty}\frac{u}{(1+u)^2}\6u=\infty,
\]
and hence $\Psi_R$ must have the claimed form.
\end{proof}

\begin{lemma}\label{ligninger}
Let $\alpha$ be a strictly positive number, and consider
  the functions $Q_\alpha$, $P_\alpha$ and $f_\alpha$ given in
  \eqref{eq3.1}, \eqref{eq3.14} and \eqref{eq3.15}. We then have

\begin{enumerate}[i]

\item For any $\rho$ in $(0,\infty)$ the equation:
\[
Q_{\alpha}(x)=\rho
\]
has at most two solutions in $(-c_{\alpha},\infty)$.

\item For any $\rho$ in $(0,\infty)$ the equation:
\[
f_\alpha(\xi)=\rho
\]
has at most two solutions in $(s_{\alpha},\infty)$.

\end{enumerate}
\end{lemma}

\begin{proof} \ (i) \ Let $\rho$ be a strictly positive number, and
assume that there exist three distinct points $x_1,x_2,x_3$ in
$(-c_{\alpha},\infty)$ such that 
\[
\rho=Q_\alpha(x_j)=-\im\Big(\frac{1}{x_j+\ri v_\alpha(x_j)}\Big),
\quad(j=1,2,3).
\]
It is elementary to check that the points $z$ in
$\C\setminus\{0\}$, for which $-\im(1/z)=\rho$, constitute the circle
$C_\rho$ in $\C$ with center $\frac{1}{2\rho}\ri$ and radius
$\frac{1}{2\rho}$ (except for the origin). Thus our assumption is
that $C_\rho$ intersects the set $\CG_\alpha'$ (given in \eqref{eq3.4a}) at
three distinct points. Note that
\[
C_\rho=\big\{\tfrac{1}{2\rho}(\ri+\e^{\ri\beta})\bigm|\beta\in(-\pi,\pi]\big\}
=\big\{\tfrac{1}{2\rho}(\cos(\beta)+\ri(1+\sin(\beta))
\bigm|\beta\in(-\pi,\pi]\big\}.
\]
Writing a point $\frac{1}{2\rho}(\cos(\beta)+\ri(1+\sin(\beta))$ from
$C_\rho\setminus\{0\}$ in polar coordinates $r\e^{\ri\theta}$ ($r>0$,
$\theta\in(0,\pi)$), it follows that
\[
r\sin(\theta)=\tfrac{1}{2\rho}(1+\sin(\beta)), \qand
r^2=\tfrac{1}{4\rho^2}(\cos^2(\beta)+1+\sin^2(\beta)+2\sin(\beta))
=\tfrac{1}{2\rho^2}(1+\sin(\beta)),
\]
so that
\[
r=\frac{1+\sin(\beta)}{2\rho^2r}=\frac{r\sin(\theta)}{\rho
  r}=\frac{1}{\rho}\sin(\theta).
\]
Hence,
\[
C_\rho=\big\{\tfrac{1}{\rho}\sin(\theta)\e^{\ri\theta}
\bigm|\theta\in(0,\pi]\big\},
\]
and our assumption thus implies that there are three distinct points
$\theta_1,\theta_2,\theta_3$ in $(0,\pi)$, such that
$\frac{1}{\rho}\sin(\theta_j)\e^{\ri\theta_j}\in\CG_\alpha'$,
$j=1,2,3$. According to \eqref{eq3.4a} and \eqref{eq3.6}, this means
that the equation
\begin{equation}
F(\tfrac{1}{\rho}\sin(\theta)\e^{\ri\theta})=\frac{1}{\alpha}
\label{eq3.3}
\end{equation}
has (at least) three distinct solutions in $(0,\pi)$. However,
Lemma~\ref{key lemma} asserts that the function
\[
\Phi_{1/\rho}(\theta)=F(\tfrac{1}{\rho}\sin(\theta)\e^{\ri\theta}),
\qquad(\theta\in(0,\pi)),
\]
is strictly decreasing on $(0,\theta_{1/\rho}]$ and strictly
increasing on $[\theta_{1/\rho},\pi)$ for some $\theta_{1/\rho}$ in
  $(0,\pi)$. Hence the equation \eqref{eq3.3} has at most two
  solutions in $(0,\pi)$, and we have reached the desired
  contradiction.

(ii) \ Let $\rho$ be a strictly positive number, and assume that there
  exist three distinct $\xi_1,\xi_2,\xi_3$ in $(s_\alpha,\infty)$
    such that $f_\alpha(\xi_j)=\rho$, $j=1,2,3$. Then there exist three
    distinct points $x_1,x_2,x_3$ in $(-c_\alpha,\infty)$, such that
    $P_\alpha(x_j)=\xi_j$, $j=1,2,3$, and it follows from formula
    \eqref{eq3.15} that
\[
\rho=f_\alpha(P_{\alpha}(x_j))=\tfrac{1}{\pi}Q_\alpha(x_j), \quad(j=1,2,3).
\]
This contradicts (i), and the proof is completed. 
\end{proof}

\begin{theorem}
For each $\alpha$ in $(0,\infty)$ the density $f_\alpha$ of the free
Gamma distribution $\nu_\alpha$ is unimodal. In fact, there
exists a number $\omega_\alpha$ in $(s_\alpha,\infty)$ such
that $f_\alpha$ is strictly increasing on
$[s_\alpha,\omega_\alpha]$ and strictly decreasing on
$[\omega_\alpha,\infty)$.
\end{theorem}

\begin{proof} The proof is an elementary consequence of
  Lemma~\ref{ligninger}(ii), but for completeness we provide the
  details.
We know that that $f_\alpha$ is continuous, that $f_\alpha(\xi)>0$
whenever $\xi>s_\alpha$, and that
\[
f_\alpha(s_\alpha)=0=\lim_{\xi\to\infty}f_\alpha(\xi)
\]
(cf.\ Lemma~\ref{intro v-alpha}(iv), Theorem~\ref{thm abs kont} and
Proposition~\ref{asympt opfoersel af taethed}).
In particular it follows that $f_\alpha$ attains a strictly positive
maximum at some point $\omega_\alpha$ in
$(s_\alpha,\infty)$. We show next that $f_\alpha$ is
non-decreasing on $[s_\alpha,\omega_\alpha]$. Indeed,
if this was not the case, we could choose $\xi_1,\xi_2$
in $(s_\alpha,\omega_\alpha)$ such that
\[
\xi_1<\xi_2, \qand f_\alpha(\xi_1)>f_\alpha(\xi_2).
\]
Choosing an arbitrary number $\rho$ in $(f_\alpha(\xi_2),f_\alpha(\xi_1))$,
it follows then by continuity of $f_\alpha$ that there must exist
$s_1$ in $(s_\alpha,\xi_1)$, $s_2$ in $(\xi_1,\xi_2)$ and $s_3$
in $(\xi_2,\omega_\alpha)$ such that 
\[
f_\alpha(s_i)=\rho, \quad (i=1,2,3).
\]
Since this contradicts Lemma~\ref{ligninger}(ii), we conclude that
$f_\alpha$ is non-decreasing on
$[s_\alpha,\omega_\alpha]$. This further implies that
$f_\alpha$ is strictly increasing on that same interval, since
otherwise $f_\alpha$ would be constant on a non-degenerate
sub-interval, which is precluded by Lemma~\ref{ligninger}(ii).

Similar (symmetric) arguments show that $f_\alpha$ is strictly
decreasing on $[\omega_\alpha,\infty)$, and this completes the proof.
\end{proof}

\section{Asymptotic behavior as $\alpha\to0$}\label{sec_alpha_mod_0}

In this section we study the asymptotic behavior of the free Gamma
distributions $\nu_\alpha$, as $\alpha\downarrow0$. We start by
considering convergence in moments.

\begin{proposition} The measures $\frac{1}{\alpha}\nu_{\alpha}$
  converge in moments to the measure
  $t^{-1}\e^{-t}1_{(0,\infty)}(t)\6t$ as $\alpha\downarrow0$. 
More precisely we have for any $p$ in $\N$ that
\[
\frac{1}{\alpha}\int_0^{\infty}t^p\,\nu_\alpha(\d t)\longrightarrow
\int_0^{\infty}t^{p-1}\e^{-t}\6t, \quad\text{as $\alpha\downarrow0$}.
\]
\end{proposition}

\begin{proof}
It follows from Proposition~\ref{asympt opfoersel af taethed} that
$\nu_\alpha$ has moments of all orders (cf.\ also \cite{BG}). It
follows moreover from \cite[Lemma~6.5]{An} that for all $p$ in $\N$
the free cumulant $r_p(\alpha)$ of $\nu_\alpha$ equals the classical
cumulant $c_p(\alpha)$ of $\mu_\alpha$ (the classical Gamma
distribution with parameter $\alpha$). The latter cumulants may be
identified by considering the Taylor expansion at 0 of
$\log(\hat{\mu}_{\alpha}(u))$. Using dominated convergence, it follows
that for any $u$ in $(-1,1)$ we have that (cf.\ \eqref{eq3.19a})
\begin{equation*}
\log(\hat{\mu}_\alpha(u))=\alpha\int_0^{\infty}
\big(\e^{\ri ut}-1\big)\frac{\e^{-t}}{t}\6t
=\alpha\int_0^{\infty}\Big(\sum_{p=1}^{\infty}\frac{\ri^pu^pt^{p-1}}{p!}\Big)
\e^{-t}\6t
=\alpha\sum_{p=1}^{\infty}\frac{\ri^p(p-1)!}{p!}u^p,
\end{equation*}
from which we may deduce that
\[
r_p(\alpha)=c_p(\alpha)=\alpha(p-1)!\quad\text{for all $p$ in $\N$.}
\]
Using the Moment-Cumulant Formula (cf.\ \cite{NiSp}) it follows further
that the $p$'th moment $m_p(\alpha)$ of $\nu_\alpha$ is given by
\[
m_p(\alpha)=r_p(\alpha)+
\sum_{k=2}^p\frac{1}{k}\binom{p}{k-1}
\sum_{q_1,\ldots,q_k\ge1\atop q_1+\cdots+q_k=p}
r_{q_1}(\alpha)r_{q_2}(\alpha)\cdots r_{q_k}(\alpha)
\]
for any $p$ in $\N$. In particular we see that $m_p(\alpha)$ is a
polynomial in $\alpha$ of degree $p$ with no constant term and linear
term $\alpha(p-1)!$. For any $p$ in $\N$ we may thus conclude that
\[
\tfrac{1}{\alpha}\int_0^{\infty}t^p\,\nu_\alpha(\d t)
=\tfrac{1}{\alpha}m_p(\alpha)\underset{\alpha\to0}{\longrightarrow}
(p-1)!=\int_0^{\infty}t^{p-1}\e^{-t}\6t,
\]
as desired.
\end{proof}

We show next that the densities of $\frac{1}{\alpha}\nu_{\alpha}$
actually converge point-wise to $t^{-1}\e^{-t}1_{(0,\infty)}(t)$ as
$\alpha\downarrow0$.

\begin{lemma}\label{lemmaB} Consider the functions $P_\alpha$ defined
  in \eqref{eq3.13}.

\begin{enumerate}[i]

\item For any $x$ in $(0,\infty)$ we have that $P_\alpha(x)\to
  x$, as $\alpha\downarrow0$.

\item For any $y$ in $(0,\infty)$ we have that
  $P_\alpha^\brinv(y)\to y$, as $\alpha\downarrow0$.

\end{enumerate}
\end{lemma}

\begin{proof} (i) \ Let $x$ be a fixed number in $(0,\infty)$. From
  \eqref{eq3.12} and \eqref{eq3.13} it follows that
\[
P_\alpha(x)=x+\ri
v_\alpha(x)+\alpha+\alpha\int_0^{\infty}
\frac{t\e^{-t}}{x-t+\ri v_\alpha(x)}\6t, \qquad(\alpha\in(0,\infty)).
\]
Lemma~\ref{egensk v-alpha}(iii) clearly implies that $v_\alpha(x)\to0$, as
$\alpha\to0$, and hence it suffices to show that
\begin{equation}
\alpha\int_0^{\infty}\frac{t\e^{-t}}{x-t+\ri v_\alpha(x)}\6t
\longrightarrow0, \quad\text{as $\alpha\to0$.}
\label{addeq7}
\end{equation}
From Lemma~\ref{egensk v-alpha}(iii) it follows furthermore that we may choose
$\alpha_1$ in $(0,\infty)$, such that
$\frac{v_\alpha(x)}{\alpha}\ge\frac{\pi}{2}x\e^{-x}$, whenever
$\alpha\in(0,\alpha_1]$. Then for all $t$ in $(0,\infty)$ and $\alpha$
in $(0,\alpha_1]$ we have that
\begin{equation}
\alpha\Big|\frac{t\e^{-t}}{x-t+\ri v_\alpha(x)}\Big|
\le\frac{t\e^{-t}}{v_\alpha(x)/\alpha}
\le\frac{t\e^{-t}}{\frac{\pi}{2}x\e^{-x}}
=\tfrac{2}{\pi}x^{-1}\e^xt\e^{-t}.
\label{addeq6}
\end{equation}
For any $t$ in $(0,\infty)\setminus\{x\}$ we note further that
\begin{equation}
\alpha\Big|\frac{t\e^{-t}}{x-t+\ri v_\alpha(x)}\Big|
\le\alpha\frac{t\e^{-t}}{|x-t|}
\longrightarrow0, \quad\text{as $\alpha\to0$.}
\label{addeq8}
\end{equation}
Combining \eqref{addeq6} and \eqref{addeq8} it follows by
dominated convergence that \eqref{addeq7} holds, as
desired.

(ii) \ Let $y$ in $(0,\infty)$ and $\epsilon$ in $(0,y)$ be
given. From (i) we know that $P_\alpha(y-\epsilon)\to y-\epsilon$,
and $P_\alpha(y+\epsilon)\to y+\epsilon$, as $\alpha\to0$. Hence we
may choose $\alpha_2$ in $(0,\infty)$ such that
\begin{equation*}
P_\alpha(y-\epsilon)<y, \qand
P_\alpha(y+\epsilon)>y, \quad\text{whenever
  $\alpha\in(0,\alpha_2]$}.
\end{equation*}
Then for any $\alpha$ in $(0,\alpha_2]$ we have that
\begin{equation}
y\in[P_\alpha(y-\epsilon),P_\alpha(y+\epsilon)]
=P_\alpha\big([y-\epsilon,y+\epsilon]\big),
\label{addeq9}
\end{equation}
since $P_\alpha$ is increasing and continuous. It follows from
\eqref{addeq9} that
\[
P_\alpha^\brinv(y)\in[y-\epsilon,y+\epsilon], \quad\text{whenever
  $\alpha\in(0,\alpha_2]$},
\]
and since $\epsilon$ was chosen arbitrarily in $(0,y)$, this
establishes (ii).
\end{proof}

\begin{proposition}\label{konv af taethed for alpha->0}
For any $x$ in $(0,\infty)$ we have that
\[
\tfrac{1}{\alpha}f_{\alpha}(x)\to x^{-1}\e^{-x}, \quad\text{as
  $\alpha\to0$.}
\]
\end{proposition}

\begin{proof} Let $x$ be a fixed number in $(0,\infty)$, and note that  
Lemma~\ref{egensk v-alpha}(iii) implies that
$\frac{v_\alpha(x)}{\alpha}\to\pi x\e^{-x}$, as $\alpha\to0$. 
Using then formula \eqref{eq3.15} we find that
\begin{equation*}
\tfrac{1}{\alpha}f_{\alpha}(P_\alpha(x))
=\frac{v_\alpha(x)/\alpha}{\pi(x^2+v_\alpha(x)^2)}
\longrightarrow\frac{x\e^{-x}}{x^2+0}=x^{-1}\e^{-x},
\quad\text{as $\alpha\to0$}.
\end{equation*}
It suffices thus to show that
\[
\frac{1}{\alpha}\big|f_{\alpha}(P_{\alpha}(x))-f_{\alpha}(x)\big|
\longrightarrow0, 
\quad\text{as $\alpha\to0$.}
\]
For all positive $\alpha$ we put $y_\alpha:=P_\alpha^\brinv(x)$, and
Lemma~\ref{lemmaB}(ii) then asserts that $y_\alpha\to x$, as
$\alpha\to0$. Given any number $\delta$ in $(0,x)$ we may
thus choose $\alpha_1$ in $(0,\infty)$ such that
\begin{equation}
y_\alpha\in[\delta,\infty), \quad\text{whenever $\alpha\in(0,\alpha_1]$}.
\label{addeq10}
\end{equation}
For any $\alpha$ in $(0,\alpha_1]$ we find then by application of
\eqref{eq3.15} that
\begin{equation}
\begin{split}
\frac{1}{\alpha}\big|f_{\alpha}(P_\alpha(x))&-f_{\alpha}(x)\big|
=\frac{1}{\alpha}
\big|f_{\alpha}(P_\alpha(x))-f_{\alpha}(P_\alpha(y_\alpha))\big|
\\[.2cm]
&=\frac{1}{\pi}\Big|\frac{v_\alpha(x)/\alpha}{x^2+v_\alpha(x)^2}
-\frac{v_\alpha(y_\alpha)/\alpha}{y_\alpha^2+v_\alpha(y_\alpha)^2}\Big|
\\[.2cm]
&\le
\frac{1}{\pi}
\Big|\frac{v_\alpha(x)/\alpha-v_\alpha(y_\alpha)/\alpha}{x^2+v_\alpha(x)^2}\Big|
+\frac{v_\alpha(y_\alpha)}{\pi\alpha}
\Big|\frac{1}{x^2+v_\alpha(x)^2}-
\frac{1}{y_\alpha^2+v_\alpha(y_\alpha)^2}\Big|.
\label{addeq13}
\end{split}
\end{equation}
Consider now in addition an arbitrary number $\gamma$ in $(0,1)$.
By Lemma~\ref{egensk v-alpha}(iii) we may then choose $\alpha_2$ in
$(0,\alpha_1]$, such that
\begin{equation}
\sup_{u\in[\delta,\infty)}
\Big|\frac{v_\alpha(u)}{\alpha}-\pi u\e^{-u}\Big|\le\gamma, 
\quad\text{whenever $\alpha\in(0,\alpha_2]$}.
\label{addeq11}
\end{equation}
Using \eqref{addeq10} and \eqref{addeq11} we find that
\begin{equation}
\frac{v_\alpha(y_\alpha)}{\alpha}\le \pi y_\alpha\e^{-y_\alpha}+\gamma
\le \pi\sup_{u\in(0,\infty)}u\e^{-u}+1<\infty,
\quad\text{whenever $\alpha\in(0,\alpha_2]$}.
\label{addeq12}
\end{equation}
Together with the fact that $y_\alpha\to x$ as $\alpha\to0$, this
implies that
\[
\frac{1}{y_\alpha^2+v_\alpha(y_\alpha)^2}
=\frac{1}{y_\alpha^2+\alpha^2(\frac{v_\alpha(y_\alpha)}{\alpha})^2}
\longrightarrow\frac{1}{x^2}, \quad\text{as $\alpha\to0$}.
\]
Since also $\frac{1}{x^2+v_\alpha(x)^2}\to\frac{1}{x^2}$, as
$\alpha\to0$, another application of \eqref{addeq12} then yields that
\begin{equation*}
\frac{v_\alpha(y_\alpha)}{\alpha}
\Big|\frac{1}{x^2+v_\alpha(x)^2}-
\frac{1}{y_\alpha^2+v_\alpha(y_\alpha)^2}\Big|\longrightarrow0,
\quad\text{as $\alpha\to0$.}
\end{equation*}
In view of \eqref{addeq13} it remains thus to show that
\begin{equation}
\Big|\frac{v_\alpha(x)/\alpha-v_\alpha(y_\alpha)/\alpha}
{x^2+v_\alpha(x)^2}\Big|
\longrightarrow0, \quad\text{as $\alpha\to0$}.
\label{addeq14}
\end{equation}
For this, note that for any $\alpha$ in $(0,\alpha_2]$ we have by
new applications of \eqref{addeq10} and \eqref{addeq11} that
\[
\Big|\frac{v_\alpha(x)/\alpha-v_\alpha(y_\alpha)/\alpha}
{x^2+v_\alpha(x)^2}\Big|
\le\frac{2\gamma+\pi|x\e^{-x}-y_\alpha\e^{-y_\alpha}|}{x^2}.
\]
Since $u\mapsto ue^{-u}$ is continuous at $x$, we may choose
$\alpha_3$ in $(0,\alpha_2]$ such that
$|x\e^{-x}-y_\alpha\e^{-y_\alpha}|\le\pi^{-1}\gamma$, whenever
$\alpha\in(0,\alpha_3]$, and then
\[
\Big|\frac{v_\alpha(x)/\alpha-v_\alpha(y_\alpha)/\alpha}
{x^2+v_\alpha(x)^2}\Big|
\le\frac{3\gamma}{x^2}, \quad\text{whenever $\alpha\in(0,\alpha_3]$.}
\]
Since $\gamma$ was chosen arbitrarily in $(0,1)$, this verifies
\eqref{addeq14} and completes the proof.
\end{proof}

\appendix
\section{Proofs of Lemmas~\ref{egensk v-alpha} and \ref{egensk
    P-alpha}}\label{tekniske beviser}

In this appendix we provide detailed (but rather technical) proofs of
Lemma~\ref{egensk v-alpha} and Lemma~\ref{egensk P-alpha}. We
start with the following preparatory result: 

\begin{lemma}\label{egensk v-alpha II}
Let $\alpha$ be a positive number and consider the
  function $v_\alpha\colon\R\to[0,\infty)$ given by
      \eqref{eq3.6a}-\eqref{eq3.6}. We then have

\begin{enumerate}[i]

\item If $0<\epsilon<x$, then
\[
v_\alpha(x)\ge
  2\alpha(x-\epsilon)\e^{-x-\epsilon}\arctan(\tfrac{\epsilon}{v_{\alpha}(x)}).
\]

\item For any $\epsilon$ in $(0,1)$ we have for all sufficiently
  large $x$ that
\[
v_\alpha(x)\le\frac{2\alpha(x+\epsilon)}{1-\epsilon}
\e^{-x+\epsilon}\arctan(\tfrac{\epsilon}{v_\alpha(x)}).
\]
\end{enumerate}
\end{lemma}

\begin{proof} (i) \ Recall first (cf.\ \eqref{eq3.6}) that
\begin{equation}
\frac{1}{\alpha}=\int_0^{\infty}\frac{t\e^{-t}}{(x-t)^2+v_\alpha(x)^2}\6t, 
\qquad(x\in(-c_\alpha,\infty)),
\label{eq3.22}
\end{equation}
Assume next that $0<\epsilon<x$. Using \eqref{eq3.22}
and the change of variables $u=\frac{t-x}{v_{\alpha}(x)}$, we find that
\begin{equation}
\begin{split}
\frac{1}{\alpha}&\ge
\int_{x-\epsilon}^{x+\epsilon}\frac{t\e^{-t}}{(t-x)^2+v_{\alpha}(x)^2}\6t 
\ge\frac{(x-\epsilon)\e^{-x-\epsilon}}{v_{\alpha}(x)^2}
\int_{x-\epsilon}^{x+\epsilon}\frac{1}{1+(\frac{t-x}{v_{\alpha}(x)})^2}\6t
\\[.2cm]
&=\frac{(x-\epsilon)\e^{-x-\epsilon}}{v_{\alpha}(x)^2}
\int_{-\epsilon/v_\alpha(x)}^{\epsilon/v_\alpha(x)}\frac{1}{1+u^2}v_\alpha(x)\6u
=\frac{2(x-\epsilon)\e^{-x-\epsilon}}{v_{\alpha}(x)}
\arctan(\tfrac{\epsilon}{v_\alpha(x)}),
\label{eq3.22a}
\end{split}
\end{equation}
from which the desired estimate follows immediately. 

(ii) \ Let $\epsilon$ be a given number in $(0,1)$, and note that for
any $t$ in $(0,\infty)$ and $x$ in $(\epsilon,\infty)$,
\[
\frac{t\e^{-t}}{(t-x)^2+v_\alpha(x)^2}1_{[0,x-\epsilon]}(t), \ 
\frac{t\e^{-t}}{(t-x)^2+v_\alpha(x)^2}1_{[x+\epsilon,\infty)}(t)
\le\epsilon^{-2}t\e^{-t}.
\]
Hence, by dominated convergence,
\[
\int_{0}^{x-\epsilon}\frac{t\e^{-t}}{(t-x)^2+v_{\alpha}(x)^2}\6t, \ 
\int_{x+\epsilon}^{\infty}\frac{t\e^{-t}}{(t-x)^2+v_{\alpha}(x)^2}\6t 
\longrightarrow0, \quad\text{as $x\to\infty$}.
\]
Thus, for all sufficiently large $x$ we have by \eqref{eq3.22} that
\begin{equation}
\begin{split}
(1-\epsilon)\frac{1}{\alpha}&\le
\int_{x-\epsilon}^{x+\epsilon}\frac{t\e^{-t}}{(t-x)^2+v_{\alpha}(x)^2}\6t 
\le\frac{(x+\epsilon)\e^{-x+\epsilon}}{v_{\alpha}(x)^2}
\int_{x-\epsilon}^{x+\epsilon}\frac{1}{1+(\frac{t-x}{v_{\alpha}(x)})^2}\6t
\\[.2cm]
&=\frac{2(x+\epsilon)\e^{-x+\epsilon}}{v_{\alpha}(x)}
\arctan(\tfrac{\epsilon}{v_\alpha(x)}),
\label{eq3.22b}
\end{split}
\end{equation}
which yields the desired estimate.
\end{proof}

\begin{proofof}[Proof of Lemma~\ref{egensk v-alpha}.] \ (i) \ Consider
    the function 
  $\tilde{F}\colon\R\times(0,\infty)\to\R$ given by
\[
\tilde{F}(x,y)=F(x+\ri y)
=\int_0^\infty\frac{t\e^{-t}}{(x-t)^2+y^2}\6t,
\qquad(x\in\R, \ y>0).
\]
Using formula \eqref{eq3.21} in the case $\alpha=1$, it follows that
\[
\tilde{F}(x,y)=1-y^{-1}\im\big(H_1(x+\ri y)\big),
\qquad((x,y)\in\R\times(0,\infty)),
\]
and since the imaginary part of an analytic function is
analytic (as a function of two real variables), we may conclude
from this that $\tilde{F}$ is analytic on $\R\times(0,\infty)$.
By differentiation under the integral sign we find in particular that
\[
\frac{\partial\tilde{F}}{\partial x}(x,y)
=-2\int_0^\infty\frac{(x-t)t\e^{-t}}{((x-t)^2+y^2)^2}\6t,
\]
and
\[
\frac{\partial\tilde{F}}{\partial y}(x,y)
=-2y\int_0^\infty\frac{t\e^{-t}}{((x-t)^2+y^2)^2}\6t.
\]
Since $v_\alpha(x)>0$ and
$\tilde{F}(x,v_{\alpha}(x))=\frac{1}{\alpha}$ for all $x$ in 
$(-c_\alpha,\infty)$, and since $\frac{\partial\tilde{F}}{\partial
  y}(x,y)<0$ for all $(x,y)$ in $\R\times(0,\infty)$,
it follows then from the implicit function theorem (for analytic
functions; see \cite[Theorem~7.6]{FG}), that $v_\alpha$ is
analytic on $(-c_{\alpha},\infty)$ with derivative given by
\begin{equation}
v_\alpha'(x)=
\frac{-\frac{\partial\tilde{F}}{\partial x}(x,v_\alpha(x))}
{\frac{\partial\tilde{F}}{\partial y}(x,v_\alpha(x))}
=\frac{-\int_0^\infty\frac{(x-t)t\e^{-t}}{((x-t)^2+v_\alpha(x)^2)^2}\6t}
{v_\alpha(x)\int_0^\infty\frac{t\e^{-t}}{((x-t)^2+v_\alpha(x)^2)^2}\6t},
\qquad(x\in(-c_\alpha,\infty)).
\label{eq4.6}
\end{equation}
In particular $v_\alpha$ is continuous on $(-c_\alpha,\infty)$. From the
defining relations \eqref{eq3.22} and \eqref{eq3.10}
it is standard to check that $v_\alpha(x)\to0$ as
$x\downarrow-c_\alpha$. Thus, $v_\alpha$ is continuous at $-c_\alpha$
as well and hence on all of $\R$.

(ii) \ Using Lemma~\ref{egensk v-alpha II}(i) we find for any positive
$\epsilon$ that 
\begin{equation*}
\liminf_{x\to\infty}\frac{v_{\alpha}(x)}{x\e^{-x}}
\ge\liminf_{x\to\infty}
\frac{2\alpha(x-\epsilon)\e^{-x-\epsilon}\arctan(\tfrac{\epsilon}{v_\alpha(x)})}
{x\e^{-x}}
=2\alpha\e^{-\epsilon}\frac{\pi}{2}
=\e^{-\epsilon}\alpha\pi,
\end{equation*}
where we have used that $v_\alpha(x)\to0$ as $x\to\infty$
(cf.\ Lemma~\ref{egensk v-alpha II}(ii)). Letting then $\epsilon\to0$,
it follows that 
\begin{equation}
\liminf_{x\to\infty}\frac{v_{\alpha}(x)}{x\e^{-x}}\ge\alpha\pi.
\label{eq3.7}
\end{equation}
Using Lemma~\ref{egensk v-alpha II}(ii) we find similarly for
$\epsilon$ in $(0,1)$ that 
\begin{equation*}
\limsup_{x\to\infty}\frac{v_{\alpha}(x)}{x\e^{-x}}
\le\limsup_{x\to\infty}
\frac{2\alpha(x+\epsilon)\e^{-x+\epsilon}\arctan(\tfrac{\epsilon}{v_\alpha(x)})}
{(1-\epsilon)x\e^{-x}}=\frac{\e^{\epsilon}}{1-\epsilon}\alpha\pi,
\end{equation*}
and letting then $\epsilon\to0$, we conclude that
\begin{equation}
\limsup_{x\to\infty}\frac{v_{\alpha}(x)}{x\e^{-x}}
\le\alpha\pi.
\label{eq3.8}
\end{equation}
Combining \eqref{eq3.7} and \eqref{eq3.8} completes the proof of (ii).

(iii) \ Let $\epsilon$ be a fixed number in
  $(0,\frac{1}{2}]$. For any $x$ in $[\epsilon,\infty)$, we note then
  that
\begin{equation*}
\begin{split}
\alpha\Big(\int_0^{x-\epsilon}&\frac{t\e^{-t}}{(x-t)^2+v_\alpha(x)^2}\6t
+\int_{x+\epsilon}^{\infty}\frac{t\e^{-t}}{(x-t)^2+v_\alpha(x)^2}\6t\Big)
\\[.2cm]
&\le\alpha\epsilon^{-2}
\Big(\int_0^{x-\epsilon}t\e^{-t}\6t
+\int_{x+\epsilon}^{\infty}t\e^{-t}\6t\Big)
\le \alpha\epsilon^{-2}\int_0^{\infty}t\e^{-t}\6t=\alpha\epsilon^{-2},
\end{split}
\end{equation*}
and thus
\[
\sup_{x\in[\epsilon,\infty)}
\alpha\Big(\int_0^{x-\epsilon}\frac{t\e^{-t}}{(x-t)^2+v_\alpha(x)^2}\6t
+\int_{x+\epsilon}^{\infty}\frac{t\e^{-t}}{(x-t)^2+v_\alpha(x)^2}\6t\Big)
\longrightarrow0, \quad\text{as $\alpha\to0$}.
\]
In combination with \eqref{eq3.22} this implies that we may choose
$\alpha_1$ in $(0,\infty)$ such that for all $\alpha$ in
$(0,\alpha_1]$ and all $x$ in $[\epsilon,\infty)$ we have that
\begin{equation*}
\begin{split}
1-\epsilon&\le\alpha
\int_{x-\epsilon}^{x+\epsilon}\frac{t\e^{-t}}{(x-t)^2+v_\alpha(x)^2}\6t
\\[.2cm]
&=2\alpha(x+\epsilon)\e^{-x+\epsilon}v_\alpha(x)^{-1}
\arctan\big(\tfrac{\epsilon}{v_\alpha(x)}\big)
\le\pi\alpha(x+\epsilon)\e^{-x+\epsilon}v_\alpha(x)^{-1},
\end{split}
\end{equation*}
where we have re-used the calculation \eqref{eq3.22b}.
Hence, it follows that
\begin{equation}
\frac{v_\alpha(x)}{\alpha}
\le\frac{\pi(x+\epsilon)\e^{-x+\epsilon}}{1-\epsilon}
\quad\text{for all $x$ in $[\epsilon,\infty)$ and $\alpha$ in
  $(0,\alpha_1]$}.
\label{addeq2}
\end{equation}
Since $\epsilon\le\frac{1}{2}$, we find in particular for all $\alpha$
in $(0,\alpha_1]$ that
\begin{equation}
\sup_{x\in[\epsilon,\infty)}v_\alpha(x)\le K_\epsilon\alpha,
\quad\text{where}\quad
K_\epsilon:=2\pi\sqrt{\e}\sup_{x\in[\epsilon,\infty)}(x+\tfrac{1}{2})\e^{-x}
<\infty.
\label{addeq1}
\end{equation}
For any $x$ in $[\epsilon,\infty)$ and $\alpha$ in $(0,\infty)$, we
note next that
\begin{equation*}
\begin{split}
1&\ge\alpha\int_{x-\epsilon}^{x+\epsilon}\frac{t\e^{-t}}{(x-t)^2+v_\alpha(x)^2}\6t
\ge\alpha(x-\epsilon)\e^{-x-\epsilon}v_\alpha(x)^{-2}
\int_{x-\epsilon}^{x+\epsilon}\frac{1}{1+(\frac{t-x}{v_\alpha(x)})^2}\6t
\\[.2cm]
&=2\alpha(x-\epsilon)\e^{-x-\epsilon}v_\alpha(x)^{-1}
\arctan\big(\tfrac{\epsilon}{v_\alpha(x)}\big).
\end{split} 
\end{equation*}
In combination with \eqref{addeq1} this shows that for all $x$ in
$[\epsilon,\infty)$ and $\alpha$ in $(0,\alpha_1]$ we have that
\[
\frac{v_\alpha(x)}{\alpha}\ge2(x-\epsilon)\e^{-x-\epsilon}\arctan(\epsilon
K_\epsilon^{-1}\alpha^{-1}).
\]
Hence, we may choose $\alpha_2$ in $(0,\alpha_1]$, such that
\begin{equation}
\frac{v_\alpha(x)}{\alpha}\ge2(x-\epsilon)\e^{-x-\epsilon}
(1-\epsilon)\tfrac{\pi}{2}
=(1-\epsilon)\pi(x-\epsilon)\e^{-x-\epsilon}
\label{addeq3}
\end{equation}
for all $x$ in $[\epsilon,\infty)$ and $\alpha$ in $(0,\alpha_2]$.
Combining now \eqref{addeq2} and \eqref{addeq3}, it follows for any
$\alpha$ in $(0,\alpha_2]$ that
\begin{equation}
\begin{split}
\sup_{x\in[\epsilon,\infty)}&\Big|\frac{v_\alpha(x)}{\alpha}-\pi
x\e^{-x}\Big|
\\[.2cm]
&\le\pi\sup_{x\in[\epsilon,\infty)}\Big[
\big(x\e^{-x}-(1-\epsilon)(x-\epsilon)\e^{-x-\epsilon}\big)
\vee
\big((1-\epsilon)^{-1}(x+\epsilon)\e^{-x+\epsilon}-x\e^{-x}\big)\Big].
\label{addeq5}
\end{split}
\end{equation}
Using that the function $x\mapsto x\e^{-x}$ is bounded on
$(0,\infty)$, it is standard to check that
\begin{equation}
\sup_{x\in(0,\infty)}
\big(x\e^{-x}-(1-\epsilon)(x-\epsilon)\e^{-x-\epsilon}\big),
\sup_{x\in(0,\infty)}
\big((1-\epsilon)^{-1}(x+\epsilon)\e^{-x+\epsilon}-x\e^{-x}\big)
\underset{\epsilon\to0}{\longrightarrow}0. 
\label{addeq4}
\end{equation}
To complete the proof, assume that positive numbers $\delta$ and
$\gamma$ are given. By \eqref{addeq4} we may then choose $\epsilon$ in
$(0,\delta\wedge\frac{1}{2}]$ such that the right hand side of
\eqref{addeq5} is smaller than $\gamma$. Applying the above
considerations to this $\epsilon$, it follows that we may choose
$\alpha_2$ in $(0,\infty)$, such that
\[
\sup_{x\in[\delta,\infty)}\Big|\frac{v_\alpha(x)}{\alpha}-\pi
x\e^{-x}\Big|\le
\sup_{x\in[\epsilon,\infty)}\Big|\frac{v_\alpha(x)}{\alpha}-\pi
x\e^{-x}\Big|\le\gamma,
\]
whenever $\alpha\in(0,\alpha_2]$.
\end{proofof}

For the proof of Lemma~\ref{egensk P-alpha} we need the following
auxiliary result.

\begin{lemma}\label{asympt lemma for P-alpha}

Let $\alpha,r,\epsilon $ be positive numbers such that
$\epsilon<1$. We then have 

\begin{enumerate}[i]

\item 
$x\displaystyle{\int_{(0,\infty)
\setminus(x-\frac{\epsilon}{x},x+\frac{\epsilon}{x})}
\frac{t^r\e^{-t}}{(x-t)^2+v_\alpha(x)^2}\6t\longrightarrow0}$,
\quad as $x\to\infty$.

\item If $0<\epsilon<x$, then
\begin{equation*}
\begin{split}
\frac{2(x-\epsilon)^r\e^{-x-\epsilon}}{v_\alpha(x)}
\arctan(\tfrac{\epsilon}{v_\alpha(x)})
&\le
\int_{x-\epsilon}^{x+\epsilon}\frac{t^r\e^{-t}}{(x-t)^2+v_\alpha(x)^2}\6t
\\[.2cm]&\le
\frac{2(x+\epsilon)^r\e^{-x+\epsilon}}{v_\alpha(x)}
\arctan(\tfrac{\epsilon}{v_\alpha(x)}).
\end{split}
\end{equation*}

\item For all sufficiently large positive $x$ we have that
\[
\frac{(x-\epsilon)(x^2-\epsilon)^2\e^{-\frac{2\epsilon}{x}}}
{\alpha x^2(x^2+\epsilon)}
\le\int_{x-\frac{\epsilon}{x}}^{x+\frac{\epsilon}{x}}
\frac{t^2\e^{-t}}{(x-t)^2+v_\alpha(x)^2}\6t
\le\frac{(x^2+\epsilon)^2\e^{\frac{2\epsilon}{x}}}
{\alpha x(x^2-\epsilon)}.
\]
\end{enumerate}
\end{lemma}

\begin{proof}
(i) \ We note first that for
 $x$ in, say, $(2,\infty)$, we have that
\begin{equation*}
\int_{x+\frac{\epsilon}{x}}^{2x}\frac{t^r\e^{-t}}{(x-t)^2+v_\alpha(x)^2}\6t
\le (2x)^r\e^{-x-\frac{\epsilon}{x}}\int_{x+\frac{\epsilon}{x}}^{2x}
\frac{1}{(\frac{\epsilon}{x})^2}\6t
=2^r\epsilon^{-2}x^r(x^3-\epsilon x)\e^{-x-\frac{\epsilon}{x}},
\end{equation*}
and similarly that
\begin{equation*}
\begin{split}
\int_{x/2}^{x-\frac{\epsilon}{x}}\frac{t^r\e^{-r}}{(x-t)^2+v_\alpha(x)^2}\6t
\le (x-\tfrac{\epsilon}{x})^r\e^{-x/2}\int_{x/2}^{x-\frac{\epsilon}{x}}
\frac{1}{(\frac{\epsilon}{x})^2}\6t
=\epsilon^{-2}(x-\tfrac{\epsilon}{x})^r(\tfrac{x^3}{2}-\epsilon x)\e^{-x/2}.
\end{split}
\end{equation*}
Moreover,
\begin{equation*}
\int_{2x}^{\infty}\frac{t^r\e^{-r}}{(x-t)^2+v_\alpha(x)^2}\6t
\le\int_{2x}^{\infty}\frac{t^r\e^{-r}}{x^2}\6t
=\frac{1}{x^2}\int_{2x}^{\infty}t^r\e^{-r}\6t,
\end{equation*}
and
\begin{equation*}
\int_{0}^{x/2}\frac{t^r\e^{-r}}{(x-t)^2+v_\alpha(x)^2}\6t
\le\int_{0}^{x/2}\frac{t^r\e^{-r}}{(\frac{x}{2})^2}\6t
=\frac{4}{x^2}\int_{0}^{x/2}t^r\e^{-r}\6t.
\end{equation*}
Now, the sum of the left hand sides of the 4 estimates above is equal
to the integral
$\int_{(0,\infty)\setminus(x-\frac{\epsilon}{x},x+\frac{\epsilon}{x})}
\frac{t^r\e^{-t}}{(x-t)^2+v_\alpha(x)^2}\6t$, and the sum of the right
hand sides is clearly of size $o(x^{-1})$ as $x\to\infty$. This shows
(i).

(ii) \ Assume that $0<\epsilon<x$. Arguing as in the proof of
Lemma~\ref{egensk v-alpha}(ii), we find that
\begin{equation*}
\begin{split}
\int_{x-\epsilon}^{x+\epsilon}\frac{t^r\e^{-t}}{(t-x)^2+v_{\alpha}(x)^2}\6t 
&\le\frac{(x+\epsilon)^r\e^{-x+\epsilon}}{v_{\alpha}(x)^2}
\int_{x-\epsilon}^{x+\epsilon}\frac{1}{1+(\frac{t-x}{v_{\alpha}(x)})^2}\6t
\\[.2cm]
&=\frac{2(x+\epsilon)^r\e^{-x+\epsilon}}{v_{\alpha}(x)}
\arctan(\tfrac{\epsilon}{v_\alpha(x)}),
\end{split}
\end{equation*}
which proves the second estimate in (ii). The first estimate follows
similarly.

\item[(iii)] Considering $x$ in $(1,\infty)$, we find by application
  of (ii) and Lemma~\ref{egensk v-alpha II}(i) (with $\epsilon$ replaced
  by $\epsilon/x$) that
\begin{equation*}
\begin{split}
\int_{x-\frac{\epsilon}{x}}^{x+\frac{\epsilon}{x}}
\frac{t^2\e^{-t}}{(x-t)^2+v_\alpha(x)^2}\6t
&\le\frac{2(x+\frac{\epsilon}{x})^2\e^{-x+\frac{\epsilon}{x}}}{v_\alpha(x)}
\arctan(\tfrac{\epsilon}{xv_\alpha(x)})
\\[.2cm]
&\le\frac{2(x+\frac{\epsilon}{x})^2\e^{-x+\frac{\epsilon}{x}}
\arctan(\tfrac{\epsilon}{xv_\alpha(x)})}
{2\alpha(x-\frac{\epsilon}{x})\e^{-x-\frac{\epsilon}{x}}
\arctan(\tfrac{\epsilon}{xv_{\alpha}(x)})}
=\frac{(x^2+\epsilon)^2\e^{\frac{2\epsilon}{x}}}
{\alpha x(x^2-\epsilon)},
\end{split}
\end{equation*}
which proves the second estimate in (iii). Regarding the first
estimate, we note first that it follows from (i) that
\begin{equation*}
\int_{(0,\infty)\setminus(x-\frac{\epsilon}{x},x+\frac{\epsilon}{x})}
\frac{t\e^{-t}}{(x-t)^2+v_\alpha(x)^2}\6t\le\frac{\epsilon}{\alpha x}
\end{equation*}
for all sufficiently large $x$, and hence by \eqref{eq3.22} and (ii)
\begin{equation*}
\frac{(1-\tfrac{\epsilon}{x})}{\alpha}\le
\int_{x-\frac{\epsilon}{x}}^{x+\frac{\epsilon}{x}}
\frac{t\e^{-t}}{(x-t)^2+v_\alpha(x)^2}\6t
\le\frac{2(x+\frac{\epsilon}{x})\e^{-x+\frac{\epsilon}{x}}}{v_\alpha(x)}
\arctan(\tfrac{\epsilon}{xv_\alpha(x)})
\end{equation*}
for all sufficiently large $x$. For such $x$ we may thus conclude that
\begin{equation*}
v_\alpha(x)
\le\frac{2\alpha(x+\frac{\epsilon}{x})
\e^{-x+\frac{\epsilon}{x}}}{1-\frac{\epsilon}{x}}
\arctan(\tfrac{\epsilon}{xv_\alpha(x)}),
\end{equation*}
which in combination with (ii) yields that
\begin{equation*}
\begin{split}
\int_{x-\frac{\epsilon}{x}}^{x+\frac{\epsilon}{x}}
\frac{t^2\e^{-t}}{(x-t)^2+v_\alpha(x)^2}\6t
&\ge
\frac{2(x-\frac{\epsilon}{x})^2\e^{-x-\frac{\epsilon}{x}}}{v_\alpha(x)}
\arctan(\tfrac{\epsilon}{xv_\alpha(x)})
\\[.2cm]&\ge
\frac{2(1-\frac{\epsilon}{x})(x-\frac{\epsilon}{x})^2
\e^{-x-\frac{\epsilon}{x}}\arctan(\tfrac{\epsilon}{xv_\alpha(x)})}
{2\alpha(x+\frac{\epsilon}{x})
\e^{-x+\frac{\epsilon}{x}}\arctan(\tfrac{\epsilon}{xv_\alpha(x)})}
\\[.2cm]&=
\frac{(x-\epsilon)(x^2-\epsilon)^2\e^{-\frac{2\epsilon}{x}}}
{\alpha x^2(x^2+\epsilon)},
\end{split}
\end{equation*}
for all sufficiently large $x$. This completes the proof.
\end{proof}

\begin{proofof}[Proof of Lemma~\ref{egensk P-alpha}.] 
(i) \ Since $H_\alpha$ is analytic on
  $\C\setminus[0,\infty)$ and $v_\alpha$ is analytic on
    $\R\setminus\{-c_\alpha\}$, it 
    follows immediately from \eqref{eq3.13} that $P_\alpha$ is
    analytic on $\R\setminus\{-c_\alpha\}$. Since $v_\alpha$ is continuous
    on $\R$, it follows also that so is $P_\alpha$.

(ii) \ For $x$ in $(-\infty,-c_\alpha)$, formula \eqref{eq3.14}
    follows immediately from \eqref{eq3.12}, since
    $v_\alpha(x)=0$. For $x$ in 
    $[-c_\alpha,\infty)$ we find, using Lemma~\ref{intro v-alpha}(iii),
    \eqref{eq3.12} and \eqref{eq3.22}, that 
\begin{equation*}
\begin{split}
P_\alpha(x)&=H_\alpha(x+\ri v_\alpha(x))
=\re\big(H_\alpha(x+\ri v_\alpha(x))\big)
\\[.2cm]
&=x+\alpha+
\alpha\int_0^\infty\re\Big(\frac{t\e^{-t}}{x+\ri
  v_\alpha(x)-t}\Big)\6t
\\[.2cm]
&=x+\alpha+
\alpha\int_0^\infty\frac{(x-t)t\e^{-t}}{(x-t)^2+v_\alpha(x)^2}\6t
\\[.2cm]
&=x+\alpha+
\alpha x\int_0^\infty\frac{t\e^{-t}}{(x-t)^2+v_\alpha(x)^2}\6t
-\alpha\int_0^\infty\frac{t^2\e^{-t}}{(x-t)^2+v_\alpha(x)^2}\6t
\\[.2cm]
&=2x+\alpha-\alpha\int_0^\infty\frac{t^2\e^{-t}}{(x-t)^2+v_\alpha(x)^2}\6t,
\end{split}
\end{equation*}
as desired.

(iii) \ Considering the function $H_\alpha$ restricted to
$(-\infty,0)$, it follows from \eqref{eq3.12} and dominated
convergence that
\[
\lim_{z\to-\infty\atop z\in\R}H_\alpha(z)=-\infty, \qand
\lim_{z\to0\atop z\in(-\infty,0)}H_\alpha(z)=0.
\]
From \eqref{eq3.9} it follows further that
\[
H_\alpha'(z)>0, \ \text{if $z\in(-\infty,-c_\alpha)$}, \qand
H_\alpha'(z)<0, \ \text{if $z\in(-c_\alpha,0)$}.
\]
Since $P_\alpha=H_\alpha$ on $(-\infty,-c_\alpha]$, we deduce from these
observations that $P_\alpha$ is strictly increasing on
$(-\infty,-c_\alpha]$, and that
$s_\alpha>\inf_{z\in(-c_\alpha,0)}H_\alpha(z)=0$.

(iv) \ Using formula \eqref{eq3.14} as well as (i) and (iii) of
  Lemma~\ref{asympt lemma for P-alpha} we find for any $\epsilon$ in
  $(0,1)$ that
\begin{equation*}
\begin{split}
\limsup_{x\to\infty}\big(x+\alpha-P_\alpha(x)\big)
&=\limsup_{x\to\infty}\Big(-x+\alpha\int_0^{\infty}
\frac{t^2\e^{-t}}{(x-t)^2+v_\alpha(x)^2}\6t\Big)
\\[.2cm]&=
\limsup_{x\to\infty}\Big(-x+\alpha\int_{x-\frac{\epsilon}{x}}^{x+\frac{\epsilon}{x}}
\frac{t^2\e^{-t}}{(x-t)^2+v_\alpha(x)^2}\6t\Big)
\\[.2cm]&\le
\limsup_{x\to\infty}\Big(-x+
\alpha\frac{(x^2+\epsilon)^2\e^{\frac{2\epsilon}{x}}}
{\alpha x(x^2-\epsilon)}\Big)
\\[.2cm]&=
\limsup_{x\to\infty}\Big(-x+
\tfrac{(x^2+\epsilon)}{x}\Big(1+\tfrac{2\epsilon}{x^2-\epsilon}\Big)
\e^{\frac{2\epsilon}{x}}\Big)
\\[.2cm]&=
\limsup_{x\to\infty}\Big(x(\e^{\frac{2\epsilon}{x}}-1)+\Big(
\tfrac{2\epsilon x}{x^2-\epsilon}+\tfrac{\epsilon}{x}+
\tfrac{2\epsilon^2}{x(x^2-\epsilon)}\Big)\e^{\frac{2\epsilon}{x}}\Big)
\\[.2cm]&=
\lim_{x\to\infty}\frac{\e^{\frac{2\epsilon}{x}}-1}{\frac{1}{x}-0}
=2\epsilon.
\end{split}
\end{equation*}
Arguing similarly we find next that
\begin{equation*}
\begin{split}
\liminf_{x\to\infty}\big(x+\alpha-P_\alpha(x)\big)
&\ge
\liminf_{x\to\infty}\Big(-x+
\alpha\frac{(x-\epsilon)(x^2-\epsilon)^2\e^{-\frac{2\epsilon}{x}}}
{\alpha x^2(x^2+\epsilon)}\Big),
\\[.2cm]&=
\liminf_{x\to\infty}\Big(-x+
\tfrac{(x-\epsilon)(x^2-\epsilon)}{x^2}
\Big(1-\tfrac{2\epsilon}{x^2+\epsilon}\Big)
\e^{-\frac{2\epsilon}{x}}\Big)
\\[.2cm]&=
\liminf_{x\to\infty}\Big(x(\e^{-\frac{2\epsilon}{x}}-1)
-\tfrac{2\epsilon x}{x^2+\epsilon}\e^{-\frac{2\epsilon}{x}}
+\tfrac{\epsilon^2-x\epsilon-x^2\epsilon}{x^2}
\Big(1-\tfrac{2\epsilon}{x^2+\epsilon}\Big)\e^{-\frac{2\epsilon}{x}}\Big)
\\[.2cm]&=
-3\epsilon.
\end{split}
\end{equation*}
Combining the two estimates obtained above, and letting
$\epsilon\to0$, we obtain (i).

(v) \  From \eqref{eq3.14} and (iv), it is clear that
$P_\alpha(x)\to\pm\infty$ as $x\to\pm\infty$, and since $P_\alpha$ is
continuous, it suffices thus to prove that $P_\alpha'(x)>0$ for all
$x$ in $\R\setminus\{-c_\alpha\}$.
In the proof of (iii) we already noted that $P_\alpha'(x)>0$ for all
$x$ in $(-\infty,-c_\alpha)$. For $x$ in $(-c_\alpha,\infty)$ we find
by differentiation in \eqref{eq3.13} that 
\begin{equation}
\begin{split}
P'_\alpha(x)&=H'_\alpha(x+\ri v_\alpha(x))(1+\ri v'_\alpha(x))
\\[.2cm]
&=\re\big(H'_\alpha(x+\ri v_\alpha(x))\big)
-\im\big(H'_\alpha(x+\ri v_\alpha(x))\big)v'_\alpha(x),
\label{eq4.2}
\end{split}
\end{equation}
where we have used that $P'_\alpha(x)\in\R$, so that
\begin{equation}
0=\im(P'_\alpha(x))=\re\big(H'_\alpha(x+\ri v_\alpha(x))\big)v'_\alpha(x)
+\im\big(H'_\alpha(x+\ri v_\alpha(x))\big).
\label{eq4.3}
\end{equation}
According to Lemma~\ref{H' ikke 0}, $\re(H'_\alpha(x+\ri
v_\alpha(x)))>0$, and hence \eqref{eq4.3} implies that
\[
v'_\alpha(x)=\frac{-\im\big(H'_\alpha(x+\ri
  v_\alpha(x))\big)}{\re\big(H'_\alpha(x+\ri v_\alpha(x))\big)},
\]
which inserted into \eqref{eq4.2} yields that
\[
P'_\alpha(x)=\re\big(H'_\alpha(x+\ri v_\alpha(x))\big)
+\frac{\im\big(H'_\alpha(x+\ri v_\alpha(x))\big)^2}
{\re\big(H'_\alpha(x+\ri v_\alpha(x))\big)} 
=\frac{|H'_\alpha(x+\ri v_\alpha(x))|^2}
{\re\big(H'_\alpha(x+\ri v_\alpha(x))\big)}>0,
\]
as desired.
\end{proofof}

\vspace{1.5cm}

\begin{minipage}[c]{0.5\textwidth}
Department of Mathematical Sciences\\
University of Copenhagen\\
Universitetsparken~5\\
2100 Copenhagen {\O}\\
Denmark\\
{\tt haagerup@math.ku.dk}
\end{minipage}
\hfill
\begin{minipage}[c]{0.5\textwidth}
Department of Mathematics\\
University of Aarhus\\
Ny Munkegade 118\\
8000 Aarhus C\\
Denmark\\
{\tt steenth@imf.au.dk}
\end{minipage}

\end{document}